\documentclass{amsart}
\usepackage{amsfonts}

\theoremstyle{plain}
\newtheorem{theorem}{Theorem}
\newtheorem{lemma}{Lemma}
\newtheorem{corollary}{Corollary}
\newtheorem{proposition}{Proposition}

\theoremstyle{definition}

\theoremstyle{remark}
\newtheorem{remark}{Remark}

\numberwithin{equation}{section}
\input{tcilatex}

\begin{document}
\title[Kibble--Slepian kernel ]{Towards a $q-$analogue of the
Kibble--Slepian formula in $3$ dimensions}
\author{Pawe\l\ J. Szab\l owski}
\address{Department of Mathematics and Information Sciences,\\
Warsaw University of Technology\\
pl. Politechniki 1, 00-661 Warsaw, Poland}
\email{pawel.szablowski@gmail.com}
\date{October, 2010}

\begin{abstract}
We study a generalization of the Kibble--Slepian (KS) expansion formula in 3
dimensions. The generalization is obtained by replacing the Hermite
polynomials by the $q-$Hermite ones. If such a replacement would lead to
non-negativity for all allowed values of parameters and for all values of
variables ranging over certain Cartesian product of compact intervals then
we would deal with a generalization of the $3$ dimensional Normal
distribution. We show that this is not the case. We indicate some values of
the parameters and some compact set in $\mathbb{R}^{3}$ of positive measure,
such that the values of the extension of KS formula are on this set
negative. Nevertheless we indicate other applications of so generalized KS
formula. Namely we use it to sum certain kernels built of the
Al-Salam--Chihara polynomials for the cases that were not considered by
other authors. One of such kernels sums up to the Askey--Wilson density
disclosing its new, interesting properties. In particular we are able to
obtain a generalization of the $2$ dimensional Poisson--Mehler formula. As a
corollary we indicate some new interesting properties of the Askey--Wilson
polynomials with complex parameters. We also pose several open questions.
\end{abstract}

\thanks{The author is grateful to the unknown referee for indicating
positions of the literature were somewhat similar but much lengthier results
were obtained. }
\maketitle

\section{Introduction}

In 1945 W.F. Kibble \cite{Kibble} and later, independently D. Slepian \cite%
{Slepian72} have extended the Poisson--Mehler formula to higher dimensions,
expanding ratio of the standardized multidimensional Gaussian density
divided by the product of one dimensional marginal densities in a multiple
sum involving only the powers of constants (correlation coefficients) and
the Hermite polynomials. The symmetry of this beautiful formula encourages
further generalizations in the sense that the Hermite polynomials in the KS
formula are substituted by their generalizations.

Nice generalization of the Hermite polynomials emerged more than hundred
years ago but only recently was intensively studied. The generalized Hermite
polynomials are called the $q-$Hermite polynomials and constitute a one
parameter family of orthogonal polynomials that for $q\allowbreak
=\allowbreak 1$ is exactly equal to the family of the classical Hermite
polynomials.

We indicate that the function of $3$ variables obtained with a help of so
generalized KS formula has many properties of a $3$ dimensional density. Let
us call it $f_{3D}.$ We know its marginals which are nonnegative. We are
able to calculate all moments of a supposed to exist random vector that
would have this function as its joint density. In particular we can
calculate the variance--covariance matrix of this "random vector". This
matrix is exactly the same as in the Gaussian case. Thus the question
remains if function $f_{3D}$ is really nonnegative for some values of
parameter $q$ (say $-1<q\leq 1)$ and all values of correlation coefficients
that make the variance-covariance matrix positive definite and almost all
values of variables from certain product of compact intervals. It will turn
out that it is not. We will indicate particular values of $q,$ of
correlation coefficients and a subset in $\mathbb{R}^{3}$ of positive
measure such that in $3$ dimensions so constructed generalization of the KS
formula is negative on this set.

Nevertheless we point out that it is worth to study described above sums
since we obtain a nice and simple tool for examining properties of different
kernels involving the so called Al-Salam--Chihara (ASC) polynomials. The
kernels built of the ASC polynomials were studied by Askey, Rahman and
Suslov in \cite{suslov96}. Such kernels have many applications particularly
in connection with certain models of the so called $q-$oscillators
considered in quantum physics. See e.g. \cite{Floreanini97}, \cite{Flo97}.

By studying the described above sums of the $q-$Hermite or ASC polynomials
we obtain kernels that are different than those considered in \cite{suslov96}%
. Hence we obtain new results related to an important problem of summing
kernels.

One of these new results is a generalization of Poisson--Mehler expansion
formula in the sense that the $q-$Hermite polynomials are replaced by the
ASC ones. Of course the sum is different. Instead of the density of measure
that makes ASC polynomials orthogonal (classical case) we get the density of
measure that makes the Askey--Wilson polynomials orthogonal. We also analyze
other non-symmetric kernels built of the ASC polynomials and sum them. As a
by-product we point out in Remark \ref{AWdensity}, below the possibility of
an interesting decomposition of the Askey-Wilson polynomials.

The paper is organized as follows. In the next Section \ref{aux} we
introduce all necessary auxiliary information concerning the so called $q-$%
series theory. In particular we introduce the $q-$Hermite and ASC
polynomials and present their basic properties. In the following Section \ref%
{main} we present our main results while less interesting or longer proofs
are shifted to the last Section \ref{dowody}. We also include special
Section \ref{open} with open problems since not all questions that appeared
when studying this beautiful object we were able to answer.

\section{Notation and auxiliary results\label{aux}}

Let us introduce notation traditionally used in the $q-$series theory. $q$
is a parameter. It can be real or complex, usually if complex, then $%
\left\vert q\right\vert <1$. We will assume however throughout the paper
that $-1<q\leq 1$. Having $q$ we define $\left[ 0\right] _{q}\allowbreak
=\allowbreak 0,$ $\left[ n\right] _{q}\allowbreak =\allowbreak 1+q+\ldots
+q^{n-1}\allowbreak ,$ $\left[ n\right] _{q}!\allowbreak =\allowbreak
\prod_{i=1}^{n}\left[ i\right] _{q},$ with $\left[ 0\right] _{q}!\allowbreak
=1,\allowbreak $%
\begin{equation*}
\QATOPD[ ] {n}{k}_{q}\allowbreak =\allowbreak \left\{ 
\begin{array}{ccc}
\frac{\left[ n\right] _{q}!}{\left[ n-k\right] _{q}!\left[ k\right] _{q}!} & 
, & n\geq k\geq 0 \\ 
0 & , & otherwise%
\end{array}%
\right. .
\end{equation*}
\newline
It is useful to use the so called $q-$Pochhammer symbol for $n\geq 1:\left(
a;q\right) _{n}=\prod_{i=0}^{n-1}\left( 1-aq^{i}\right) ,$ with $\left(
a;q\right) _{0}=1$ , $\left( a_{1},a_{2},\ldots ,a_{k};q\right)
_{n}\allowbreak =\allowbreak \prod_{i=1}^{k}\left( a_{i};q\right) _{n}$.%
\newline
Often $\left( a;q\right) _{n}$ as well as $\left( a_{1},a_{2},\ldots
,a_{k};q\right) _{n}$ will be abbreviated to $\left( a\right) _{n}$ and $%
\left( a_{1},a_{2},\ldots ,a_{k}\right) _{n},$ if it will not cause
misunderstanding.

It is easy to notice that $\left( q\right) _{n}=\left( 1-q\right) ^{n}\left[
n\right] _{q}!$ and that%
\begin{equation*}
\QATOPD[ ] {n}{k}_{q}\allowbreak =\allowbreak \allowbreak \left\{ 
\begin{array}{ccc}
\frac{\left( q\right) _{n}}{\left( q\right) _{n-k}\left( q\right) _{k}} & ,
& n\geq k\geq 0 \\ 
0 & , & otherwise%
\end{array}%
\right. .
\end{equation*}%
\newline
Notice that for $n\allowbreak \geq \allowbreak k\allowbreak \geq \allowbreak
0$ we have: $\left[ n\right] _{1}\allowbreak =\allowbreak n,\left[ n\right]
_{1}!\allowbreak =\allowbreak n!,$ $\QATOPD[ ] {n}{k}_{1}\allowbreak
=\allowbreak \binom{n}{k}$ (binomial coefficient)$,$ $\left( a;1\right)
_{n}\allowbreak =\allowbreak \left( 1-a\right) ^{n}$ and $\left[ n\right]
_{0}\allowbreak =1$ for $n\geq 1$, $\left[ n\right] _{0}!\allowbreak
=\allowbreak 1,$ $\QATOPD[ ] {n}{k}_{0}\allowbreak =\allowbreak 1,$ $\left(
a;0\right) _{n}\allowbreak =\allowbreak \left\{ 
\begin{array}{ccc}
1 & if & n=0 \\ 
1-a & if & n\geq 1%
\end{array}%
\right. .$ Let us also denote $I_{A}(x)\allowbreak =\allowbreak \left\{ 
\begin{array}{ccc}
1 & if & x\in A \\ 
0 & if & x\notin A%
\end{array}%
\right. .$

To define briefly and swiftly the one-dimensional distributions that later
will be used to construct possible multidimensional generalizations of the
Normal distributions, let us define the following sets: 
\begin{equation*}
S\left( q\right) =[-2/\sqrt{1-q},2/\sqrt{1-q}],\text{for }\left\vert
q\right\vert <1\text{ and }S\left( 1\right) =\mathbb{R}.
\end{equation*}

Further we define the following sets of polynomials:

-the $q-$Hermite polynomials defined by the relationship:%
\begin{equation}
H_{n+1}(x|q)=xH_{n}(x|q)-[n]_{q}H_{n-1}(x|q),  \label{He}
\end{equation}%
for $n\geq 0$ with $H_{-1}(x|q)=0,$ $H_{0}(x|q)=1,$

-the Al-Salam--Chihara (ASC) polynomials defined by the relationship$:$%
\begin{equation}
P_{n+1}(x|y,\rho ,q)=(x-\rho yq^{n})P_{n}(x|y,\rho ,q)-(1-\rho
^{2}q^{n-1})[n]_{q}P_{n-1}(x|y,\rho ,q),  \label{AlSC}
\end{equation}%
for $n\geq 0$ with $P_{-1}\left( x|y,\rho ,q\right) \allowbreak =\allowbreak
0,$ $P_{0}\left( x|y,\rho ,q\right) \allowbreak =\allowbreak 1$,

-the Chebyshev polynomials of the second kind defined by the relationship: 
\begin{equation*}
U_{n+1}\left( x\right) =2xU_{n}\left( x\right) -U_{n-1}\left( x\right) ,
\end{equation*}%
for $n\geq 0$ with $U_{-1}\left( x\right) \allowbreak =\allowbreak 0,$ $%
U_{0}\left( x\right) \allowbreak =\allowbreak 1.$

Notice that $H_{n}\left( x|1\right) \allowbreak =\allowbreak H_{n}\left(
x\right) ,$ $H_{n}\left( x|0\right) \allowbreak =\allowbreak U_{n}\left(
x/2\right) $ and \newline
$P_{n}\left( x|y,\rho ,1\right) \allowbreak =\allowbreak H_{n}\left( \frac{%
x-\rho y}{\sqrt{1-\rho ^{2}}}\right) \left( 1-\rho ^{2}\right) ^{n/2},$
where $\left\{ H_{n}\left( x\right) \right\} $ denote the 'probabilistic'
Hermite polynomials i.e. monic polynomials that are orthogonal with respect
to the measure with the density $\exp (-x^{2}/2)/\sqrt{2\pi }.$ \newline
From Lemma \ref{pomoc} ii), below it follows that: 
\begin{equation}
P_{n}\left( x|y,\rho ,0\right) \allowbreak =\allowbreak U_{n}\left(
x/2\right) \allowbreak -\allowbreak \rho yU_{n-1}\left( x/2\right)
\allowbreak +\allowbreak \rho ^{2}U_{n-2}\left( x/2\right) .  \label{q=0}
\end{equation}

The polynomials (\ref{He}) satisfy the following very useful identity
originally formulated for the so called continuous $q-$Hermite polynomials $%
h_{n}$ (can be found in e.g. \cite{IA} Thm. 13.1.5) related to the
polynomials $H_{n}$ by: 
\begin{equation}
h_{n}\left( x|q\right) \allowbreak =\allowbreak \left( 1-q\right)
^{n/2}H_{n}\left( \frac{2x}{\sqrt{1-q}}|q\right) ,~~n\geq 1,  \label{q-cont}
\end{equation}%
and here, below presented for the polynomials $H_{n}$ : for $n,m\geq 0$ we
have 
\begin{equation}
H_{n}\left( x|q\right) H_{m}\left( x|q\right) =\sum_{j=0}^{\min \left(
n,m\right) }\QATOPD[ ] {m}{j}_{q}\QATOPD[ ] {n}{j}_{q}\left[ j\right]
_{q}!H_{n+m-2j}\left( x|q\right) .  \label{identity}
\end{equation}%
One can also find in the literature (e.g. \cite{Carlitz56}, \cite{Chen08})
the following useful formula: for $m,n\geq 0$ we have

\begin{equation}
H_{n+m}\left( x\right) =\sum_{j=0}^{\min \left( n,m\right) }(-1)^{j}q^{%
\binom{j}{2}}\QATOPD[ ] {m}{j}_{q}\QATOPD[ ] {n}{j}_{q}\left[ j\right]
_{q}!H_{n-j}\left( x|q\right) H_{m-j}\left( x|q\right)   \label{identyty2}
\end{equation}%
that was originally formulated for the so called Rogers--Szeg\"{o}
polynomials defined by:%
\begin{equation}
W_{n}\left( x|q\right) =\sum_{j=0}^{n}\QATOPD[ ] {n}{j}_{q}x^{j},  \label{Wn}
\end{equation}%
that are related to the continuous $q-$Hermite polynomials by: 
\begin{equation*}
h_{n}\left( x|q\right) \allowbreak =\allowbreak e^{in\theta }W_{n}\left(
e^{-i2\theta }|q\right) ,
\end{equation*}%
for $n\geq 0$ with $x\allowbreak =\allowbreak \cos \theta .$ (See also \cite%
{IA}.)

To simplify notation let us introduce the following auxiliary polynomials of
order at most $2:$ 
\begin{gather}
r_{k}\left( y|q\right) =(1+q^{k})^{2}-(1-q)yq^{k},  \label{rk} \\
v_{0}\left( x,y|\rho ,q\right) =(1-\rho ^{2})^{2}-(1-q)\rho (1+\rho
^{2})xy+(1-q)\rho ^{2}(x^{2}+y^{2}), \\
v_{k}\left( x,y|\rho ,q\right) =v_{0}\left( x,y|\rho q^{k},q\right) ,\text{ }%
k\geq 1.  \label{vk}
\end{gather}%
It is known (see e.g. \cite{bryc1}) that the $q-$Hermite polynomials are
monic and orthogonal with respect to the measure that has density given by:%
\begin{equation}
f_{N}\left( x|q\right) =\frac{\left( q\right) _{\infty }\sqrt{1-q}}{2\pi 
\sqrt{r_{0}\left( x^{2}|q\right) }}\prod_{k=0}^{\infty }r_{k}\left(
x^{2}|q\right) I_{S\left( q\right) }\left( x\right) ,  \label{qN}
\end{equation}%
defined for $\left\vert q\right\vert <1,$ $x\in \mathbb{R}$ and 
\begin{equation}
f_{N}\left( x|1\right) \allowbreak =\allowbreak \frac{1}{\sqrt{2\pi }}\exp
\left( -x^{2}/2\right) .  \label{q=1}
\end{equation}

Similarly it is known (e.g. from \cite{bms}) that the ASC polynomials are
monic and orthogonal with respect to the measures that for $q\in (-1,1]$
have densities. These densities are given for $\left\vert q\right\vert <1$
by: 
\begin{subequations}
\label{fCN}
\begin{equation}
f_{CN}\left( x|y,\rho ,q\right) =\frac{\sqrt{1-q}\left( q\right) _{\infty
}\left( \rho ^{2}\right) _{\infty }}{2\pi \sqrt{r_{0}\left( x^{2}|q\right) }}%
\prod_{k=0}^{\infty }\frac{r_{k}\left( x^{2}|q\right) }{v_{k}\left( x,y|\rho
,q\right) }I_{S\left( q\right) }\left( x\right)   \label{_1}
\end{equation}%
$\left\vert \rho \right\vert <1$, $x\in \mathbb{R},$ $y\in S\left( q\right) $
and for $q\allowbreak =\allowbreak 1$ by: 
\end{subequations}
\begin{equation*}
f_{CN}\left( x|y,\rho ,1\right) \allowbreak =\allowbreak \frac{1}{\sqrt{2\pi
\left( 1-\rho ^{2}\right) }}\exp \left( -\frac{\left( x-\rho y\right) ^{2}}{%
2\left( 1-\rho ^{2}\right) }\right) .
\end{equation*}

It is known (see e.g. \cite{IA} formula 13.1.10) that for $\left\vert
q\right\vert <1:$%
\begin{equation}
\sup_{x\in S\left( q\right) }\left\vert H_{n}\left( x|q\right) \right\vert
\leq W_{n}\left( 1|q\right) \left( 1-q\right) ^{-n/2}.  \label{ogr_H}
\end{equation}%
where $W_{n}$ is given by (\ref{Wn}).

We will also need auxiliary polynomials $\left\{ B_{n}\left( x|q\right)
\right\} _{n\geq -1}$ defined by the following 3-term recurrence:%
\begin{equation}
B_{n+1}\left( y|q\right) \allowbreak =\allowbreak -q^{n}yB_{n}\left(
y|q\right) +q^{n-1}\left[ n\right] _{q}B_{n-1}\left( y|q\right) ;n\geq 0,
\label{_B}
\end{equation}%
with $B_{-1}\left( y|q\right) =0,$ $B_{0}\left( y|q\right) =1$. 

In fact the polynomials $\left\{ B_{n}\left( y|q\right) \right\} _{n\geq -1}$
are equal to the polynomials $\left\{ h_{n}\left( y|q^{-1}\right) \right\} $
scaled and normalized in a certain way. The polynomials $\left\{ h_{n}\left(
y|q^{-1}\right) \right\} $ are known for a long time and were studied in 
\cite{AI84} and \cite{Askey89}. However with the present scaling and
normalization they were introduced in \cite{bms} where some of their basic
properties were presented and their auxiliary r\^{o}le in finding connection
coefficients between ASC and q-Hermite polynomials was shown. One can show
(see e.g. \cite{bms}) that $B_{n}\left( x|1\right) \allowbreak =\allowbreak
i^{n}H_{n}\left( ix\right) $ and 
\begin{equation*}
B_{n}\left( x|0\right) \allowbreak =\allowbreak \left\{ 
\begin{array}{ccc}
1 & if & n=0\vee 2 \\ 
-x & if & n=1 \\ 
0 & if & n>2%
\end{array}%
\right. .
\end{equation*}%
Further properties of these polynomials including their relationship to the $%
q$-Hermite polynomials are presented in \cite{Szab3}.

Facts concerning the $q$-Hermite, ASC and polynomials $B_{n}$, necessary for
the derivation of the results of the paper, are collected in the following
Lemma.

\begin{lemma}
\label{pomoc}Assume that $0\allowbreak <\allowbreak q\allowbreak \leq
\allowbreak 1,$ $\left\vert \rho \right\vert ,\left\vert \rho
_{1}\right\vert ,\left\vert \rho _{2}\right\vert <1,$ $n,m\geq 0,$ $x,y,z\in
S\left( q\right) ,$ then

i) 
\begin{equation*}
H_{n}\left( x|q\right) \allowbreak =\allowbreak \sum_{k=0}^{n}\QATOPD[ ] {n}{%
k}_{q}\rho ^{n-k}H_{n-k}\left( y|q\right) P_{k}\left( x|y,\rho ,q\right) ,
\end{equation*}

ii) $\forall n\geq 1:$%
\begin{gather*}
P_{n}\left( x|y,\rho ,q\right) \allowbreak =\allowbreak \sum_{k=0}^{n}\QATOPD%
[ ] {n}{k}_{q}\rho ^{n-k}B_{n-k}\left( y|q\right) H_{k}\left( x|q\right) , \\
\sum_{j=0}^{n}\QATOPD[ ] {n}{j}B_{n-j}\left( x|q\right) H_{j}\left(
x|q\right) \allowbreak =\allowbreak 0,
\end{gather*}

iii) 
\begin{equation*}
\int_{S\left( q\right) }H_{n}\left( x|q\right) H_{m}\left( x|q\right)
f_{N}\left( x|q\right) dx\allowbreak =\allowbreak \left\{ 
\begin{array}{ccc}
0 & when & n\neq m \\ 
\left[ n\right] _{q}! & when & n=m%
\end{array}%
\right. ,
\end{equation*}

iv) 
\begin{equation*}
\int_{S\left( q\right) }P_{n}\left( x|y,\rho ,q\right) P_{m}\left( x|y,\rho
,q\right) f_{CN}\left( x|y,\rho ,q\right) dx\allowbreak =\allowbreak \left\{ 
\begin{array}{ccc}
0 & when & n\neq m \\ 
\left( \rho ^{2}\right) _{n}\left[ n\right] _{q}! & when & n=m%
\end{array}%
\right. ,
\end{equation*}

v) 
\begin{equation*}
\int_{S\left( q\right) }H_{n}\left( x|q\right) f_{CN}\left( x|y,\rho
,q\right) dx=\rho ^{n}H_{n}\left( y|q\right) ,
\end{equation*}

vi) 
\begin{equation*}
\int_{S\left( q\right) }f_{CN}\left( x|y,\rho _{1},q\right) f_{CN}\left(
y|z,\rho _{2},q\right) dy=f_{CN}\left( x|z,\rho _{1}\rho _{2},q\right) ,
\end{equation*}

vii)%
\begin{equation*}
\sum_{n=0}^{\infty }\frac{\rho ^{n}}{\left[ n\right] _{q}!}H_{n}\left(
x|q\right) H_{n}\left( y|q\right) \allowbreak =\allowbreak f_{CN}\left(
x|y,\rho ,q\right) /f_{N}\left( x|q\right) ,
\end{equation*}

viii) For $|t|,\left\vert q\right\vert <1:$%
\begin{equation*}
\sum_{i=0}^{\infty }\frac{W_{i}\left( 1|q\right) t^{i}}{\left( q\right) _{i}}%
\allowbreak =\allowbreak \frac{1}{\left( t\right) _{\infty }^{2}}%
,\sum_{i=0}^{\infty }\frac{W_{i}^{2}\left( 1|q\right) t^{i}}{\left( q\right)
_{i}}\allowbreak =\allowbreak \frac{\left( t^{2}\right) _{\infty }}{\left(
t\right) _{\infty }^{4}},
\end{equation*}%
convergence is absolute, where $W_{i}\left( x|q\right) $ is defined by (\ref%
{Wn}),

ix) 
\begin{equation*}
\forall x,y\in S\left( q\right) :0<C\left( y,\rho ,q\right) \leq \frac{%
f_{CN}\left( x|y,\rho ,q\right) }{f_{N}\left( x|q\right) }\leq \frac{\left(
\rho ^{2}\right) _{\infty }}{\left( \rho \right) _{\infty }^{4}}.
\end{equation*}
\end{lemma}

\begin{proof}
ii) was proved in \cite{bms}. i) is proved in \cite{IRS99} (formula 4.7) for
the polynomials $h_{n}$ and $Q_{n}\left( x|a,b,q\right) \allowbreak \ $%
related to the polynomials $P_{n}$ by the relationship $Q_{n}(x|a,b,q)%
\allowbreak =\allowbreak \left( 1-q\right) ^{n/2}P_{n}\left( \frac{2x}{\sqrt{%
1-q}}|\frac{2a}{\sqrt{\left( 1-q\right) b}},\sqrt{b},q\right) $. However one
can also derive it easily from ii). iii) and iv) are known (see e.g. \cite%
{IA}) for polynomials $h_{n}$ and $Q_{n}$. Thus here they are presented
after necessary adaptation to polynomials $H_{n}$ and $P_{n}.$ v) and vi)
can be found in \cite{bryc1}, but their particular cases in different form
were also shown in \cite{Bo}. vii) is in fact the famous Poisson--Mehler
formula which has many proofs. One of them is in \cite{IA} the other e.g. in 
\cite{bressoud} or \cite{Szab4}. viii) see Exercise 12.2(b) and 12.2(c) of 
\cite{IA}. ix) was proved in \cite{Szab3} (Proposition 1 vii)).
\end{proof}

Let us remark, following \cite{Szab3}, that 
\begin{equation}
f_{AW}(x|y,\rho _{1},z,\rho _{2},q)=\frac{f_{CN}\left( y|x,\rho
_{1},q\right) f_{CN}\left( x|z,\rho _{2},q\right) }{f_{CN}\left( y|z,\rho
_{1}\rho _{2},q\right) }  \label{AW}
\end{equation}%
is the density of the measure that makes the re-scaled Askey--Wilson (AW)
polynomials orthogonal. The AW polynomials mentioned in the paper are
considered for certain complex valued parameters related to $y,\rho
_{1},z,\rho _{2}$. For details see formula (2.5) of \cite{Szab3}. We will
call the function defined by (\ref{AW}) the AW density.

As mentioned in the Introduction the main object of this paper is a
generalization of Kibble--Slepian formula. Traditional KS formula has a form
of an expansion of the ratio of the non-degenerated $n-$dimensional Gaussian
density divided by the product of its one dimensional marginals in the
multiple series (in fact involving $n(n-1)/2$ fold sum) of the Hermite
polynomials in one variable with coefficients that are powers of
off-diagonal elements of variance-covariance matrix.\newline
We will analyze only the generalization of its three dimensional version. It
is simple to express and general enough to expose interesting properties and
applications.

Namely we will analyze the following function given by: 
\begin{gather}
f_{3D}(x_{1},x_{2},x_{3}|\rho _{12},\rho _{13},\rho
_{23},q)=\prod_{i=1}^{3}f_{N}\left( x_{i}|q\right)  \label{3D1} \\
\times \sum_{i,j,k\geq 0}\frac{\rho _{12}^{j}\rho _{23}^{k}\rho _{13}^{i}}{%
\left[ i\right] _{q}!\left[ j\right] _{q}!\left[ k\right] _{q}!}%
H_{i+j}\left( x_{1}|q\right) H_{j+k}\left( x_{2}|q\right) H_{i+k}\left(
x_{3}|q\right) ,  \label{3D2}
\end{gather}%
where $\left\vert \rho _{12}\right\vert ,\left\vert \rho _{13}\right\vert
,\left\vert \rho _{23}\right\vert <1.$ Let us denote%
\begin{equation*}
\mathbf{\rho =}\left[ 
\begin{array}{ccc}
1 & \rho _{12} & \rho _{13} \\ 
\rho _{12} & 1 & \rho _{23} \\ 
\rho _{13} & \rho _{23} & 1%
\end{array}%
\right] ,~~\Delta =\det \mathbf{\rho }.
\end{equation*}%
We will assume that parameters $\rho _{12},$ $\rho _{13},$ $\rho _{23}$ are
such that 
\begin{equation*}
\Delta =\allowbreak \allowbreak 1\allowbreak +\allowbreak 2\rho _{12}\rho
_{13}\rho _{23}\allowbreak -\allowbreak \rho _{12}^{2}\allowbreak
-\allowbreak \rho _{13}^{2}\allowbreak -\allowbreak \rho
_{23}^{2}\allowbreak \geq \allowbreak 0.
\end{equation*}%
Notice that parameter $\Delta $ can be presented in the following form: 
\begin{equation}
\Delta =(1-\rho _{12}^{2})(1-\rho _{23}^{2})-(\rho _{13}-\rho _{12}\rho
_{23})^{2},  \label{wyzn}
\end{equation}%
and similarly for other pairs of indices $\left( 1,2\right) $ and $(2,3)$
since $\Delta $ is obviously symmetric in $\rho _{12},\rho _{13},$ \ $\rho
_{23}.$

The above mentioned assumptions concerning the parameters $\rho _{12},\rho
_{13},$ $\rho _{23}$ will be assumed throughout the remaining part of the
paper. Similarly we will assume that $x_{i}\in S\left( q\right) ,$ $%
i\allowbreak =\allowbreak 1,2,3$ unless otherwise stated. Besides from Lemma %
\ref{pomoc} viii), ix) it follows that all considered in this paper series
for $\left\vert q\right\vert <1$ are absolutely convergent, hence we will
not repeat this statement unless it will be necessary.

\begin{remark}
Let us immediately observe that when $q\allowbreak =\allowbreak 1,$ that is
when $H_{n}\left( x|q\right) $ is substituted by $H_{n}\left( x\right) $ and 
$f_{N}\left( x_{i}|q\right) \allowbreak $ by $\allowbreak \exp \left(
-x_{i}^{2}/2\right) /\sqrt{2\pi },$ then by Kibble--Slepian formula (see
e.g. \cite{Slepian72} Example 2)) for $n\allowbreak =\allowbreak 3$ $,$ $%
f_{3D}\left( x_{1},x_{2},x_{3}|\rho _{12},\rho _{13},\rho _{23},1\right)
\allowbreak $ is the density of the Gaussian distribution $N\left( \mathbf{0}%
,\mathbf{\rho }\right) $.

Besides let us notice that the function $f_{3D}$ has the same symmetry with
respect to variables $x_{1},x_{2},x_{3},\rho _{12},\rho _{13},\rho _{23}$ as
the density of the Normal distribution $N\left( \mathbf{0},\mathbf{\rho }%
\right) .$
\end{remark}

\begin{remark}
In the literature functions somewhat similar to (\ref{3D1}-\ref{3D2}) were
considered. In \cite{is202} (formula 4.13) sums of the form $\sum_{k\geq 0}%
\frac{t^{k}}{[k]_{q}!}H_{n+k}\left( x|q\right) H_{k}\left( y|q\right)
H_{k}\left( z|q\right) $ for all $n$'s while in \cite{AIs85} (formula 2.5)
sum of the form $\sum_{m,n,k}\frac{x^{m}y^{n}z^{k}}{\left( q\right)
_{m}\left( q\right) _{n}\left( q\right) _{k}}W_{m+k}\left( a|q\right)
W_{n+k}\left( b|q\right) $ ( for $W_{n}$ see (\ref{Wn})) were analyzed. Both
these results refer to and are based on the beautiful results of Carlitz 
\cite{Carlitz72}. The obtained sums are very complicated, expressed through
the basic hypergeometric function $_{6}\phi _{5}$ or $_{3}\phi _{2}.$ These
results show that summing products of more than $2$ $q$-Hermite polynomials
is a very subtle and complicated task. On the other hand the results of
these paper dealing also with sums of products of $3$ $q$-Hermite
polynomials follow subtle path in this area leading to results of relatively
simple form. By the way our results also refer to above-mentioned results of
Carlitz.
\end{remark}

Let us present some immediate remarks concerning the function $f_{3D}.$

\begin{remark}
\label{density}i) 
\begin{equation*}
\int_{S\left( q\right) \times S\left( q\right) \times S\left( q\right)
}f_{3D}\left( x_{1},x_{2},x_{3}|\rho _{12},\rho _{13},\rho _{23},q\right)
dx_{1}dx_{2}dx_{3}\allowbreak =\allowbreak 1,
\end{equation*}

ii)%
\begin{gather*}
\int_{S\left( q\right) \times S\left( q\right) }f_{3D}\left(
x_{1},x_{2},x_{3}|\rho _{12},\rho _{13},\rho _{23},q\right)
dx_{1}dx_{2}\allowbreak =\allowbreak f_{N}\left( x_{3}|q\right) , \\
\int_{S\left( q\right) \times S\left( q\right) \times S\left( q\right)
}H_{n}\left( x_{1}|q\right) H_{m}\left( x_{2}|q\right) f_{3D}\left(
x_{1},x_{2},x_{3}|\rho _{12},\rho _{13},\rho _{23},q\right)
dx_{3}dx_{1}dx_{2} \\
=\allowbreak \left\{ 
\begin{array}{ccc}
0 & if & n\neq m \\ 
\rho _{12}^{n}\left[ n\right] _{q}! & if & m=n%
\end{array}%
\right.
\end{gather*}

and similarly for other pairs $\left( 1,3\right) $ and $\left( 2,3\right) .$

In particular we have:%
\begin{gather*}
\int_{S\left( q\right) \times S\left( q\right) \times S\left( q\right)
}x_{1}f_{3D}\left( x_{1},x_{2},x_{3}|\rho _{12},\rho _{13},\rho
_{23},q\right) dx_{1}dx_{2}dx_{3}\allowbreak =\allowbreak 0, \\
\int_{S\left( q\right) \times S\left( q\right) \times S\left( q\right)
}x_{1}^{2}f_{3D}\left( x_{1},x_{2},x_{3}|\rho _{12},\rho _{13},\rho
_{23},q\right) dx_{1}dx_{2}dx_{3}\allowbreak =\allowbreak 1
\end{gather*}%
\newline
and again similarly for the remaining indices $2$ and $3$.

iii) 
\begin{equation*}
\int_{S\left( q\right) }f_{3D}\left( x_{1},x_{2},x_{3}|\rho _{12},\rho
_{13},\rho _{23},q\right) dx_{3}\allowbreak =\allowbreak f_{CN}\left(
x_{1}|x_{2},\rho _{12},q\right) f_{N}\left( x_{2}|q\right) .
\end{equation*}

iv) If $\rho _{13}\allowbreak =\allowbreak \rho _{23}\allowbreak
=\allowbreak 0,$ then 
\begin{equation*}
f_{3D}\left( x_{1},x_{2},x_{3}|\rho _{12},\rho _{13},\rho _{23},q\right)
\allowbreak =\newline
\allowbreak f_{CN}\left( x_{1}|x_{2},\rho _{12},q\right) f_{N}\left(
x_{2}|q\right) f_{N}\left( x_{3}|q\right)
\end{equation*}
and similarly for other pairs $\left( 1,3\right) $ and $\left( 2,3\right) .$
\end{remark}

\begin{proof}
In all assertions we apply Lemma \ref{pomoc} iii), the fact that $%
H_{1}\left( x|q\right) \allowbreak =\allowbreak x$ and $H_{2}\left(
x|q\right) \allowbreak =\allowbreak x^{2}-1$ and also formulae (\ref%
{identity}) and Lemma \ref{pomoc} v). iv) follows directly from Lemma \ref%
{pomoc} iiv).
\end{proof}

As it follows from the above mentioned Remark that $f_{3D}$ is a serious
candidate for the $3$ dimensional density. It has nonnegative marginal
densities equal to the densities $f_{N}$ and $f_{N}f_{CN}.$ Moreover if $%
f_{3D}$ was a joint density of certain $3$ dimensional random vector it
would follow from Remark \ref{density} ii) that variance-covariance matrix
of this vector would be equal to $\left[ 
\begin{array}{ccc}
1 & \rho _{12} & \rho _{13} \\ 
\rho _{12} & 1 & \rho _{23} \\ 
\rho _{13} & \rho _{23} & 1%
\end{array}%
\right] .$

To analyze its properties deeper we will need the following Lemma.

\begin{lemma}
\label{UogCarl}Let us denote $\gamma _{m,k}\left( x,y|\rho ,q\right)
\allowbreak =\allowbreak \sum_{i=0}^{\infty }\frac{\rho ^{i}}{\left[ i\right]
_{q}!}H_{i+m}\left( x|q\right) H_{i+k}\left( y|q\right) \allowbreak .$

i) Then 
\begin{equation}
\gamma _{m,k}\left( x,y|\rho ,q\right) \allowbreak =\allowbreak \gamma
_{0,0}\left( x,y|\rho ,q\right) Q_{m,k}\left( x,y|\rho ,q\right) ,
\label{gamma_m_k}
\end{equation}%
where $Q_{m,k}$ is a polynomial in $x$ and $y$ of order at most $m+k.$ 
\newline
Further denote $C_{n}\left( x,y|\rho _{1},\rho _{2},\rho _{3},q\right)
\allowbreak =\allowbreak \sum_{i=0}^{n}\QATOPD[ ] {n}{i}_{q}\rho
_{1}^{n-i}\rho _{2}^{i}Q_{n-i},_{i}\left( x,y|\rho _{3,}q\right) .$ \newline
Then we have in particular:

ii) 
\begin{gather*}
Q_{k,m}\left( y,x|\rho ,q\right) \allowbreak =\allowbreak Q_{m,k}\left(
x,y|\rho ,q\right) \allowbreak \\
=\allowbreak \sum_{s=0}^{k}(-1)^{s}q^{\binom{s}{2}}\QATOPD[ ] {k}{s}_{q}\rho
^{s}H_{k-s}\left( y|q\right) P_{m+s}(x|y,\rho ,q)/(\rho ^{2})_{m+s},
\end{gather*}
for all $x,y\in S\left( q\right) $ and $q\in (-1,1],$

iii) 
\begin{equation*}
C_{n}\left( x,y|\rho _{1},\rho _{2},\rho _{3},q\right) \allowbreak
=\allowbreak \sum_{s=0}^{n}\QATOPD[ ] {n}{s}_{q}H_{n-s}\left( y|q\right)
P_{s}\left( x|y,\rho _{3},q\right) \rho _{1}^{n-s}\rho _{2}^{s}\left( \rho
_{1}\rho _{3}/\rho _{2}\right) _{s}/\left( \rho _{3}^{2}\right) _{s}.
\end{equation*}
\end{lemma}

\begin{proof}
Assertions i) and ii) are proved in \cite{Szab3}. Thus we will prove only
iii). We have $C_{n}\left( x,y|\rho _{1},\rho _{2},\rho _{3},q\right)
\allowbreak =\allowbreak \allowbreak $\newline
$\sum_{i=0}^{n}\QATOPD[ ] {n}{i}_{q}\rho _{1}^{n-i}\rho
_{2}^{i}\sum_{j=0}^{n-i}(-1)^{j}\QATOPD[ ] {n-i}{j}_{q}q^{\binom{j}{2}}\rho
_{3}^{j}H_{n-i-j}\left( y|q\right) P_{i+j}\left( x|y,\rho _{3},q\right)
/(\rho _{3}^{2})_{i+j}\allowbreak =\allowbreak $\newline
$\sum_{s=0}^{n}\QATOPD[ ] {n}{s}_{q}H_{n-s}\left( y|q\right) P_{s}\left(
x|y,\rho _{3},q\right) /\left( \rho _{3}^{2}\right) _{s}\sum_{j=0}^{s}\QATOPD%
[ ] {s}{j}_{q}(-1)^{j}q^{\binom{j}{2}}\rho _{1}^{n-s+j}\rho _{2}^{s-j}\rho
_{3}^{j}\allowbreak =\allowbreak $\newline
$\sum_{s=0}^{n}\QATOPD[ ] {n}{s}_{q}H_{n-s}\left( y|q\right) P_{s}\left(
x|y,\rho _{3},q\right) \rho _{1}^{n-s}\rho _{2}^{s}/\left( \rho
_{3}^{2}\right) _{s}\sum_{j=0}^{s}\QATOPD[ ] {s}{j}_{q}(-1)^{j}q^{\binom{j}{2%
}}(\rho _{1}\rho _{3}/\rho _{2})^{j}\allowbreak =\allowbreak $\newline
$\sum_{s=0}^{n}\QATOPD[ ] {n}{s}_{q}H_{n-s}\left( y|q\right) P_{s}\left(
x|y,\rho _{3},q\right) \rho _{1}^{n-s}\rho _{2}^{s}\left( \rho _{1}\rho
_{3}/\rho _{2}\right) _{s}/\left( \rho _{3}^{2}\right) _{s}.$ On the way we
have used the following identity $\left( a\right) _{n}\allowbreak
=\allowbreak \sum_{i=0}^{n}\QATOPD[ ] {n}{i}_{q}\left( -1\right) ^{i}q^{%
\binom{i}{2}}a^{j}.$
\end{proof}

We get immediate observations:

\begin{corollary}
\label{wn1}For all $n\geq 1$ we have:

i)%
\begin{equation*}
P_{n}\left( x|y,\rho ,q\right) \allowbreak =\allowbreak \left( \rho
^{2}\right) _{n}\sum_{i=0}^{n}\QATOPD[ ] {n}{i}_{q}\left( -1\right) ^{i}q^{%
\binom{i}{2}}\rho ^{i}H_{n-i}\left( x|q\right) P_{i}\left( y|x,\rho
,q\right) /\left( \rho ^{2}\right) _{i},
\end{equation*}

ii)%
\begin{equation*}
\int_{S\left( q\right) }P_{n}\left( x|y,\rho ,q\right) f_{CN}\left( y|x,\rho
,q\right) dy\allowbreak =\allowbreak \left( \rho ^{2}\right) _{n}H_{n}\left(
x|q\right) ,
\end{equation*}

iii) 
\begin{equation*}
C_{n}\left( x,y|\rho _{2}\rho _{3},\rho _{2},\rho _{3},q\right) \allowbreak
=\allowbreak \rho _{2}^{n}H_{n}\left( x|q\right) ,
\end{equation*}

iv) 
\begin{equation*}
C_{n}\left( x,y|\rho _{1},\rho _{1}\rho _{3},\rho _{3},q\right) \allowbreak
=\allowbreak \rho _{1}^{n}H_{n}\left( y|q\right) ,
\end{equation*}

v) 
\begin{eqnarray*}
C_{n}\left( x,y|\rho _{1},\rho _{2},0,q\right) \allowbreak &=&\allowbreak
\sum_{s=0}^{n}\QATOPD[ ] {n}{s}_{q}\rho _{1}^{n-s}\rho _{2}^{s}H_{n-s}\left(
y|q\right) H_{s}\left( x|q\right) , \\
C_{n}\left( x,y|0,\rho _{2},\rho _{3},q\right) \allowbreak &=&\allowbreak
\rho _{2}^{n}P_{n}\left( x|y,\rho _{3},q\right) /\left( \rho _{3}^{2}\right)
_{n}, \\
C_{n}\left( x,y|\rho _{1},0,\rho _{3},q\right) \allowbreak &=&\allowbreak
\rho _{1}^{n}P_{n}\left( y|x,\rho _{3},q\right) /\left( \rho _{3}^{2}\right)
_{n},
\end{eqnarray*}

vi)%
\begin{equation*}
C_{n}\left( x,y|\rho _{1},\rho _{2},\rho _{3},1\right) \allowbreak
=\allowbreak \left( \frac{\rho _{1}^{2}+\rho _{2}^{2}-2\rho _{1}\rho
_{2}\rho _{3}}{1-\rho _{3}^{2}}\right) ^{n/2}H_{n}\left( \frac{x(\rho
_{2}-\rho _{1}\rho _{3})+y(\rho _{1}-\rho _{2}\rho _{3})}{\sqrt{(1-\rho
_{3}^{2})(\rho _{1}^{2}+\rho _{2}^{2}-2\rho _{1}\rho _{2}\rho _{3})}}\right)
.
\end{equation*}
\end{corollary}

\begin{proof}
i) see Corollary 2 of \cite{Szab3} . ii) We use previous assertion and Lemma %
\ref{pomoc} iv). iii) We have $\left( \rho _{2}\rho _{3}^{2}/\rho
_{2}\right) _{s}\allowbreak \allowbreak =\allowbreak \left( \rho
_{3}^{2}\right) _{s}$, hence $C_{n}\left( x,y|\rho _{2}\rho _{3},\rho
_{2},\rho _{3},q\right) \allowbreak =\allowbreak \rho _{2}^{n}\sum_{s=0}^{n}%
\QATOPD[ ] {n}{s}_{q}H_{n-s}\left( y|q\right) P_{s}\left( x|y,\rho
_{3},q\right) \rho _{3}^{n-s}\allowbreak =\allowbreak \rho
_{2}^{n}H_{n}\left( x|q\right) $ by Lemma \ref{pomoc} i). iv) $\left( \rho
_{1}\rho _{3}/\rho _{1}\rho _{3}\right) _{s}\allowbreak =\left\{ 
\begin{array}{ccc}
1 & if & s=0 \\ 
0 & if & s>0%
\end{array}%
\right. \allowbreak $ v) First two statements are direct consequence of the
assumptions that $\rho _{3}\allowbreak =\allowbreak 0$ or $\rho
_{1}\allowbreak =\allowbreak 0.$ The third one follows the fact that $\rho
_{2}^{s}\left( \rho _{1}\rho _{3}/\rho _{2}\right) _{s}\allowbreak
=\allowbreak \prod_{i=1}^{s}\left( \rho _{2}-q^{i-1}\rho _{1}\rho
_{3}\right) ,$ which for $\rho _{2}\allowbreak =\allowbreak 0$ is gives $%
\left( -1\right) ^{s}q^{\binom{s}{2}}\rho _{1}^{s}\rho _{3}^{s}.$ Then we
apply assertion i) of this Corollary.

To get vi) first we notice that 
\begin{equation*}
P_{n}\left( x|y,\rho _{3},1\right) /\left( 1-\rho _{3}^{2}\right)
^{n}\allowbreak =\allowbreak H_{n}\left( \frac{x-\rho _{3}y}{\sqrt{1-\rho
_{3}^{2}}}\right) /\left( 1-\rho _{3}^{2}\right) ^{n/2}.
\end{equation*}%
Then we apply the well known (see e.g. \cite{AAR}) formula for addition of
Hermite polynomials :%
\begin{equation*}
\sum_{i=0}^{n}\binom{n}{i}a^{n-i}b^{i}H_{n-i}\left( \xi \right) H_{i}\left(
\theta \right) =\left( a^{2}+b^{2}\right) ^{n/2}H_{n}\left( \frac{a\xi
+b\theta }{\sqrt{a^{2}+b^{2}}}\right)
\end{equation*}%
with $a\allowbreak =\allowbreak \rho _{1},$ $b\allowbreak =\allowbreak \frac{%
\left( \rho _{2}-\rho _{1}\rho _{3}\right) }{\sqrt{1-\rho _{3}^{2}}},$ $\xi
\allowbreak =\allowbreak y$ and $\theta \allowbreak =\allowbreak \frac{%
x-\rho _{3}y}{\sqrt{1-\rho _{3}^{2}}}.$ Next we observe that $%
a^{2}+b^{2}\allowbreak =\allowbreak \frac{\rho _{1}^{2}+\rho _{2}^{2}-2\rho
_{1}\rho _{2}\rho _{3}}{1-\rho _{3}^{2}}$ and $a\xi +b\theta \allowbreak
=\allowbreak \frac{x(\rho _{2}-\rho _{1}\rho _{3})+y(\rho _{1}-\rho _{2}\rho
_{3})}{1-\rho _{3}^{2}}.$
\end{proof}

\section{Main results\label{main}}

Applying Lemma \ref{UogCarl} to the function $f_{3D}$ we have the following
Proposition:

\begin{proposition}
\label{proposition}i)%
\begin{gather}
f_{3D}(x_{1},x_{2},x_{3}|\rho _{12},\rho _{13},\rho _{23},q)=f_{CN}\left(
x_{3}|x_{1},\rho _{13},q\right) f_{N}\left( x_{1}|q\right) f_{N}\left(
x_{2}|q\right)  \label{exp in HC} \\
\times \sum_{s\geq 0}\frac{1}{\left[ s\right] _{q}!}H_{s}\left(
x_{2}|q\right) C_{s}\left( x_{1},x_{3}|\rho _{12},\rho _{23},\rho
_{13},q\right) ,  \notag
\end{gather}%
similarly for other pairs $\left( 1,3\right) $ and $\left( 2,3\right) ,$

ii) 
\begin{gather}
f_{3D}(x_{1},x_{2},x_{3}|\rho _{12},\rho _{13},\rho _{23},q)=f_{CN}\left(
x_{1}|x_{3},\rho _{13},q\right) f_{CN}\left( x_{3}|x_{2},\rho _{23},q\right)
f_{N}\left( x_{2}|q\right)  \label{exp_in_ASC} \\
\times \sum_{s=0}^{\infty }\frac{\rho _{12}^{s}\left( \rho _{13}\rho
_{23}/\rho _{12}\right) _{s}}{\left[ s\right] _{q}!\left( \rho
_{13}^{2},\rho _{23}^{2}\right) _{s}}P_{s}\left( x_{1}|x_{3},\rho
_{13},q\right) P_{s}\left( x_{2}|x_{3},\rho _{23},q\right) ,  \notag
\end{gather}%
similarly for other pairs $\left( 1,3\right) $ and $\left( 2,3\right)
.\allowbreak \allowbreak \allowbreak $
\end{proposition}

\begin{proof}
Lengthy proof is shifted to section \ref{dowody}.
\end{proof}

As an immediate corollary we have the following formula:

\begin{corollary}
\begin{gather*}
\sum_{s=0}^{\infty }\frac{\rho _{12}^{s}\left( \rho _{13}\rho _{23}/\rho
_{12}\right) _{s}}{\left[ s\right] _{q}!\left( \rho _{13}^{2},\rho
_{23}^{2}\right) _{s}}P_{s}\left( x_{1}|x_{3},\rho _{13},q\right)
P_{s}\left( x_{2}|x_{3},\rho _{23},q\right) \\
=\frac{f_{CN}\left( x_{1}|x_{2},\rho _{12},q\right) }{f_{CN}\left(
x_{1}|x_{3},\rho _{13},q\right) }\sum_{k\geq 0}\frac{\rho _{13}^{k}\left(
\rho _{12}\rho _{23}/\rho _{13}\right) _{k}}{\left[ k\right] _{q}!\left(
\rho _{13}^{2},\rho _{23}^{2}\right) _{s}}P_{k}\left( x_{1}|x_{2},\rho
_{12},q\right) P_{k}\left( x_{3}|x_{2},\rho _{23},q\right) .
\end{gather*}
\end{corollary}

\begin{proof}
From Proposition \ref{proposition} ii) we get 
\begin{gather*}
\sum_{s=0}^{\infty }\frac{\rho _{12}^{s}\left( \rho _{13}\rho _{23}/\rho
_{12}\right) _{s}}{\left[ s\right] _{q}!\left( \rho _{13}^{2},\rho
_{23}^{2}\right) _{s}}P_{s}\left( x_{1}|x_{3},\rho _{13},q\right)
P_{s}\left( x_{2}|x_{3},\rho _{23},q\right) = \\
\frac{f_{CN}\left( x_{1}|x_{2},\rho _{12},q\right) f_{CN}\left(
x_{2}|x_{3},\rho _{23},q\right) f_{N}\left( x_{3}|q\right) }{f_{CN}\left(
x_{1}|x_{3},\rho _{13},q\right) f_{CN}\left( x_{3}|x_{2},\rho _{23},q\right)
f_{N}\left( x_{2}|q\right) } \\
\times \sum_{k\geq 0}\frac{\rho _{13}^{k}\left( \rho _{12}\rho _{23}/\rho
_{13}\right) _{k}}{\left[ k\right] _{q}!\left( \rho _{13}^{2},\rho
_{23}^{2}\right) _{s}}P_{k}\left( x_{1}|x_{2},\rho _{12},q\right)
P_{k}\left( x_{3}|x_{2},\rho _{23},q\right) .
\end{gather*}%
Now we use the fact that: $f_{CN}\left( x_{2}|x_{3},\rho _{23},q\right)
f_{N}\left( x_{3}|q\right) \allowbreak =\allowbreak f_{CN}\left(
x_{3}|x_{2},\rho _{23},q\right) f_{N}\left( x_{2}|q\right) .$
\end{proof}

\begin{corollary}
If $\rho _{12}\allowbreak =\allowbreak \rho _{13}\rho _{23}$ then\newline
i) 
\begin{equation*}
f_{3D}\left( x_{1},x_{2},x_{3}|\rho _{13}\rho _{23},\rho _{13},\rho
_{23},q\right) =f_{CN}\left( x_{1}|x_{3},\rho _{13},q\right) f_{CN}\left(
x_{3}|x_{2},\rho _{23},q\right) f_{N}\left( x_{2}|q\right) ,
\end{equation*}%
ii) 
\begin{equation}
\sum_{s\geq 0}\frac{\rho _{13}^{3}}{\left[ s\right] _{q}!\left( \rho
_{13}^{2}\rho _{23}^{2}\right) _{s}}P_{s}\left( x_{1}|x_{2},\rho _{13}\rho
_{23},q\right) P_{s}\left( x_{3}|x_{2},\rho _{23},q\right) =\frac{%
f_{CN}\left( x_{1}|x_{3},\rho _{13},q\right) }{f_{CN}\left( x_{1}|x_{2},\rho
_{13}\rho _{23},q\right) }.  \label{rozwiniecie}
\end{equation}

If $\rho _{12}\allowbreak =\allowbreak 0$ then\newline
iii) 
\begin{eqnarray}
&&\sum_{s=0}^{\infty }\frac{\left( -1\right) ^{s}q^{\binom{s}{2}}\rho
_{13}^{s}\rho _{23}^{s}}{\left[ s\right] _{q}!\left( \rho _{13}^{2},\rho
_{23}^{2}\right) _{s}}P_{s}\left( x_{1}|x_{3},\rho _{13},q\right)
P_{s}\left( x_{2}|x_{3},\rho _{23},q\right)   \label{rho=0} \\
&=&\frac{f_{N}\left( x_{1}|q\right) }{f_{CN}\left( x_{1}|x_{3},\rho
_{13},q\right) }\sum_{k\geq 0}\frac{\rho _{13}^{k}}{\left[ k\right]
_{q}!\left( \rho _{23}^{2}\right) _{k}}H_{k}\left( x_{1}|q\right)
P_{k}\left( x_{3}|x_{2},\rho _{23},q\right) .  \notag
\end{eqnarray}
\end{corollary}

\begin{proof}
i) We use Corollary \ref{wn1} iii) and deduce that $\sum_{s\geq 0}\frac{1}{%
\left[ s\right] _{q}!}H_{s}\left( x_{2}|q\right) C_{s}\left(
x_{1},x_{3}|\rho _{12},\rho _{23},\rho _{13},q\right) \allowbreak
=\allowbreak \sum_{s\geq 0}\frac{\rho _{23}^{s}}{\left[ s\right] _{q}!}%
H_{s}\left( x_{2}|q\right) H_{s}\left( x_{3}|q\right) \allowbreak
=\allowbreak f_{CN}\left( x_{3}|x_{2},\rho _{23},q\right) /f_{N}\left(
x_{3}|q\right) $ (by Poisson-Mehler formula) or we notice that $\left( \rho
_{13}\rho _{23}/\rho _{12}\right) _{s}$ when $\rho _{12}\allowbreak
=\allowbreak \rho _{13}\rho _{23}$ is equal to $0$ for $s\geq 1$ and use (%
\ref{exp_in_ASC}).

ii) We use equivalent form of (\ref{exp_in_ASC}) that is%
\begin{eqnarray*}
f_{3D}(x_{1},x_{2},x_{3}|\rho _{12},\rho _{13},\rho _{23},q) &=&f_{CN}\left(
x_{1}|x_{2},\rho _{12},q\right) f_{CN}\left( x_{2}|x_{3},\rho _{23},q\right)
f_{N}\left( x_{3}|q\right)  \\
&&\times \sum_{s=0}^{\infty }\frac{\rho _{13}^{s}\left( \rho _{12}\rho
_{23}/\rho _{13}\right) _{s}}{\left[ s\right] _{q}!\left( \rho
_{13}^{2},\rho _{23}^{2}\right) _{s}}P_{s}\left( x_{1}|x_{2},\rho _{13}\rho
_{23},q\right) P_{s}\left( x_{3}|x_{2},\rho _{23},q\right) 
\end{eqnarray*}%
apply assumption, observing that then $\left( \rho _{12}\rho _{23}/\rho
_{13}\right) _{s}\allowbreak =\allowbreak \left( \rho _{23}^{2}\right) _{s}$
and using the assertion i) of this Corollary.

iii) We use the fact that $\rho _{12}^{s}\left( \rho _{13}\rho _{23}/\rho
_{12}\right) _{s}\allowbreak =\allowbreak \prod_{i=1}^{s}(\rho
_{12}-q^{i-1}\rho _{13}\rho _{23})$ which for $\rho _{12}\allowbreak
=\allowbreak 0$ equals to $\left( -1\right) ^{s}q^{1+\ldots +s-1}\rho
_{13}^{s}\rho _{23}^{s}$ and the fact that $P_{s}\left( x|y,0,q\right)
\allowbreak =\allowbreak H_{s}\left( x|q\right) .$
\end{proof}

We have also the following remark concerning the ordinary, probabilistic
Hermite polynomials

\begin{remark}
\label{gaus}We have%
\begin{gather}
\sqrt{\frac{(1-\rho _{13}^{2})(1-\rho _{23}^{2})-\left( \rho _{12}-\rho
_{13}\rho _{23}\right) ^{2}}{(1-\rho _{13}^{2})(1-\rho _{23}^{2})}}
\label{gaus2} \\
\times \sum_{s=0}^{\infty }\frac{\left( \rho _{12}-\rho _{13}\rho
_{23}\right) ^{s}}{s!(1-\rho _{13}^{2})^{s/2}(1-\rho _{23}^{2})^{s/2}}%
H_{s}\left( \frac{x_{1}-\rho _{13}x_{3}}{\sqrt{1-\rho _{13}^{2}}}\right)
H_{s}\left( \frac{x_{2}-\rho _{23}x_{3}}{\sqrt{1-\rho _{23}^{2}}}\right) 
\notag \\
=\exp \left( -\frac{1}{2}\left[ 
\begin{array}{ccc}
x_{1} & x_{2} & x_{3}%
\end{array}%
\right] \left[ \left[ 
\begin{array}{ccc}
1 & \rho _{12} & \rho _{13} \\ 
\rho _{12} & 1 & \rho _{23} \\ 
\rho _{13} & \rho _{23} & 1%
\end{array}%
\right] ^{-1}-\left[ 
\begin{array}{ccc}
1 & \rho _{13}\rho _{23} & \rho _{13} \\ 
\rho _{13}\rho _{23} & 1 & \rho _{23} \\ 
\rho _{13} & \rho _{23} & 1%
\end{array}%
\right] ^{-1}\right] \left[ 
\begin{array}{c}
x_{1} \\ 
x_{2} \\ 
x_{3}%
\end{array}%
\right] \right)  \notag
\end{gather}
\end{remark}

\begin{proof}
We use the fact that $P_{n}\left( x|y,\rho ,1\right) /\left( \rho
^{2}|1\right) _{n}\allowbreak =\allowbreak H_{n}\left( \frac{x-\rho y}{\sqrt{%
1-\rho ^{2}}}\right) /(1-\rho ^{2})^{n/2}$. Then $\sum_{s=0}^{\infty }\frac{%
\left( \rho _{12}-\rho _{13}\rho _{23}\right) ^{s}}{s!(1-\rho
_{13}^{2})^{s/2}(1-\rho _{23}^{2})^{s/2}}H_{s}\left( \frac{x_{1}-\rho
_{13}x_{3}}{\sqrt{1-\rho _{13}^{2}}}\right) H_{s}\left( \frac{x_{2}-\rho
_{23}x_{3}}{\sqrt{1-\rho _{23}^{2}}}\right) $ can be summed by the classical
Poisson--Mehler formula yielding 
\begin{equation}
\frac{1}{\sqrt{1-r^{2}}}\exp (-\frac{r^{2}(\xi ^{2}+\zeta ^{2})-2r\xi \zeta 
}{2(1-r^{2})}),  \label{pom1}
\end{equation}%
with $r\allowbreak =\allowbreak \frac{\left( \rho _{12}-\rho _{13}\rho
_{23}\right) }{\sqrt{(1-\rho _{13}^{2})(1-\rho _{23}^{2})}},$ $\xi
\allowbreak =\allowbreak \frac{x_{1}-\rho _{13}x_{3}}{\sqrt{1-\rho _{13}^{2}}%
}$ and $\zeta \allowbreak =\allowbreak \frac{x_{2}-\rho _{23}x_{3}}{\sqrt{%
1-\rho _{23}^{2}}}.$ It is a matter of elementary calculus to check that the
exponent of (\ref{pom1}) equals to the exponent of the right hand side of (%
\ref{gaus2}).
\end{proof}

\begin{remark}
Notice that we get a new proof of the classical KS formula in $3$
dimensions, since the left hand side of (\ref{exp_in_ASC}) is given by (\ref%
{3D1}-\ref{3D2}) while the right hand side can be, in view of the Remark \ref%
{gaus}, written as 
\begin{gather*}
\frac{1}{\sqrt{2\pi (1-\rho _{13}^{2})(1-\rho _{23}^{2})}}\allowbreak \times 
\newline
\allowbreak \exp (-\frac{1}{2}\left[ 
\begin{array}{ccc}
x_{1} & x_{2} & x_{3}%
\end{array}%
\right] \left[ 
\begin{array}{ccc}
1 & \rho _{13}\rho _{23} & \rho _{13} \\ 
\rho _{13}\rho _{23} & 1 & \rho _{23} \\ 
\rho _{13} & \rho _{23} & 1%
\end{array}%
\right] ^{-1}\left[ 
\begin{array}{c}
x_{1} \\ 
x_{2} \\ 
x_{3}%
\end{array}%
\right] )\times \\
\frac{1}{\sqrt{\frac{(1-\rho _{13}^{2})(1-\rho _{23}^{2})-\left( \rho
_{12}-\rho _{13}\rho _{23}\right) ^{2}}{(1-\rho _{13}^{2})(1-\rho _{23}^{2})}%
}} \\
\times \exp \left( -\frac{1}{2}\left[ 
\begin{array}{ccc}
x_{1} & x_{2} & x_{3}%
\end{array}%
\right] \left[ \left[ 
\begin{array}{ccc}
1 & \rho _{12} & \rho _{13} \\ 
\rho _{12} & 1 & \rho _{23} \\ 
\rho _{13} & \rho _{23} & 1%
\end{array}%
\right] ^{-1}-\left[ 
\begin{array}{ccc}
1 & \rho _{13}\rho _{23} & \rho _{13} \\ 
\rho _{13}\rho _{23} & 1 & \rho _{23} \\ 
\rho _{13} & \rho _{23} & 1%
\end{array}%
\right] ^{-1}\right] \left[ 
\begin{array}{c}
x_{1} \\ 
x_{2} \\ 
x_{3}%
\end{array}%
\right] \right) \allowbreak \times \\
\frac{1}{\sqrt{2\pi (1-\rho _{13}^{2})(1-\rho _{23}^{2})-\left( \rho
_{12}-\rho _{13}\rho _{23}\right) ^{2}}}\times \exp \left( -\frac{1}{2}\left[
\begin{array}{ccc}
x_{1} & x_{2} & x_{3}%
\end{array}%
\right] \left[ 
\begin{array}{ccc}
1 & \rho _{12} & \rho _{13} \\ 
\rho _{12} & 1 & \rho _{23} \\ 
\rho _{13} & \rho _{23} & 1%
\end{array}%
\right] ^{-1}\left[ 
\begin{array}{c}
x_{1} \\ 
x_{2} \\ 
x_{3}%
\end{array}%
\right] \right) ,
\end{gather*}%
which is the density of $N\left( \mathbf{0,\rho }\right) $ distribution.
Thus (\ref{exp_in_ASC}) is really a generalization of $3$ dimensional KS
formula.
\end{remark}

We have also the following observation concerning the AW density.

\begin{remark}
\label{AWdensity}Following observation (\ref{AW}) and Lemma \ref{pomoc} vii)
we deduce that 
\begin{equation*}
\frac{f_{CN}\left( x_{1}|x_{3},\rho _{13},q\right) f_{CN}\left(
x_{3}|x_{2},\rho _{23},q\right) }{f_{CN}\left( x_{1}|x_{2},\rho _{13}\rho
_{23},q\right) }\allowbreak \allowbreak =\allowbreak f_{AW}(x_{3}|x_{1},\rho
_{13},x_{2},\rho _{23},q).
\end{equation*}%
Hence we have the following expansion of the AW density:%
\begin{equation}
f_{AW}(x_{3}|x_{1},\rho _{13},x_{2},\rho _{23})=f_{CN}\left(
x_{3}|x_{2},\rho _{23},q\right) \sum_{s\geq 0}\frac{\rho _{13}^{s}}{\left[ s%
\right] _{q}!\left( \rho _{13}^{2}\rho _{23}^{2}\right) _{s}}P_{s}\left(
x_{1}|x_{2},\rho _{13}\rho _{23},q\right) P_{s}\left( x_{3}|x_{2},\rho
_{23},q\right) ,  \label{AW_exp}
\end{equation}%
which is an analogue of the Poisson--Mehler expansion formula interpreted as
the one dimensional expansion given in Lemma \ref{pomoc} vii). This result
has been obtained by other methods in \cite{Szab6}. Besides from (\ref%
{AW_exp}) it follows that: 
\begin{equation*}
\forall n\geq 1:\int_{S\left( q\right) }(P_{n}\left( x_{3}|x_{2},\rho
_{23},q\right) -\frac{\rho _{13}^{n}\left( \rho _{23}\right) _{n}}{\left(
\rho _{13}^{2}\rho _{23}^{2}\right) _{n}}P_{n}\left( x_{1}|x_{2},\rho
_{13}\rho _{23},q\right) )f_{AW}\left( x_{3}|x_{1},\rho _{13},x_{2},\rho
_{23}\right) dx_{3}=0.
\end{equation*}%
Thus there must exist functions $F_{i,n}\left( x_{1},x_{2}|\rho _{13},\rho
_{23},q\right) $ such that:%
\begin{gather*}
\forall j\geq 1:A_{j}\left( x_{3}|x_{1},\rho _{13},x_{2},\rho _{23},q\right)
= \\
\sum_{n=0}^{j}F_{j,n}\left( x_{1},x_{2}|\rho _{13},\rho _{23},q\right)
(P_{n}\left( x_{3}|x_{2},\rho _{23},q\right) -\frac{\rho _{13}^{n}\left(
\rho _{23}\right) _{n}}{\left( \rho _{13}^{2}\rho _{23}^{2}\right) _{n}}%
P_{n}\left( x_{1}|x_{2},\rho _{13}\rho _{23},q\right) ),
\end{gather*}%
where $\left\{ A_{j}\right\} _{j\geq 0}$ denote the Askey-Wilson polynomials
with complex parameters related to $x_{1},$ $\rho _{13},$ $x_{2},$ $\rho
_{23}$ by the formulae (2.17)-(2.20) of \cite{Szab3}.
\end{remark}

As far as general properties of the function $f_{3D}$ are concerned we have
the following formula that expresses function $C_{n}$ in terms of $q-$%
Hermite polynomials.

\begin{proposition}
\label{na hermity}$\forall n\geq 1:$%
\begin{gather*}
C_{n}\left( x_{1},x_{3}|\rho _{12},\rho _{23},\rho _{13},q\right) =\frac{1}{%
\left( \rho _{13}^{2}\right) _{n}}\sum_{k=0}^{\left\lfloor n/2\right\rfloor
}(-1)^{k}q^{\binom{k}{2}}\QATOPD[ ] {n}{2k}_{q}\QATOPD[ ] {2k}{k}_{q}\left[ k%
\right] _{q}!\rho _{12}^{k}\rho _{13}^{k}\rho _{23}^{k}\left( \frac{\rho
_{12}\rho _{13}}{\rho _{23}},\frac{\rho _{13}\rho _{23}}{\rho _{12}}\right)
_{k}\allowbreak  \\
\sum_{i=0}^{n-2k}\QATOPD[ ] {n-2k}{i}_{q}\rho _{23}^{i}\left( \frac{\rho
_{12}\rho _{13}}{\rho _{23}}q^{k}\right) _{i}\rho _{12}^{n-i-2k}\left( \frac{%
\rho _{13}\rho _{23}}{\rho _{12}}q^{k}\right) _{n-2k-i}H_{i}\left(
x_{1}|q\right) H_{n-2k-i}\left( x_{3}|q\right) 
\end{gather*}
\end{proposition}

\begin{proof}
Long proof is shifted to Section \ref{dowody}
\end{proof}

Let us remark that when $\rho _{13}\allowbreak =\allowbreak \rho _{12}\rho
_{23}$ then 
\begin{gather*}
C_{n}\left( x_{1},x_{3}|\rho _{12},\rho _{23},\rho _{12}\rho _{23},q\right) =%
\frac{1}{\left( \rho _{12}^{2}\rho _{23}^{2}\right) _{n}}\sum_{k=0}^{\left%
\lfloor n/2\right\rfloor }(-1)^{k}q^{\binom{k}{2}}\QATOPD[ ] {n}{2k}_{q}%
\QATOPD[ ] {2k}{k}_{q}\left[ k\right] _{q}!\rho _{12}^{2k}\rho
_{23}^{2k}\left( \rho _{12}^{2},\rho _{23}^{2}\right) _{k}\allowbreak  \\
\sum_{i=0}^{n-2k}\QATOPD[ ] {n-2k}{i}_{q}\rho _{23}^{i}\left( \rho
_{12}^{2}q^{k}\right) _{i}\rho _{12}^{n-i-2k}\left( \rho
_{23}^{2}q^{k}\right) _{n-2k-i}H_{i}\left( x_{1}|q\right) H_{n-2k-i}\left(
x_{3}|q\right) 
\end{gather*}%
the formula obtained in \cite{Szab3} (3.2-3.3) in the context of
Askey--Wilson polynomials.

As stated in the Introduction the problem of non-negativity of the function $%
f_{3D}$ for all allowed values of $q$ and $\rho ^{\prime }s$ has negative
solution. Above we indicated that if $\rho _{12}\allowbreak =\allowbreak
\allowbreak \rho _{13}\rho _{23}$ (or similarly for some other pair of
indices) then $f_{3D}$ is positive for all $-1\allowbreak <\allowbreak
q\allowbreak \leq \allowbreak 1$ and $x_{1},x_{2},x_{3}\allowbreak \in
\allowbreak S\left( q\right) .$ Now we will indicate some another
relationship between $\rho ^{\prime }s$ and $q$ so that for some values of $%
x_{1},x_{2},x_{3}\allowbreak \in \allowbreak S\left( q\right) $ the function 
$f_{3D}$ is negative.

We have the following Theorem

\begin{theorem}
\label{Main}Let $\left\vert q\right\vert <1$ and $\rho _{12}\allowbreak
=\allowbreak \allowbreak q\rho _{13}\rho _{23}$ , then there exists a set $%
S\allowbreak \subset \allowbreak S\left( q\right) \times S\left( q\right)
\times S\left( q\right) $ of positive Lebesgue measure such that for all $%
\left( x_{1},x_{2},x_{3}\right) \allowbreak \in \allowbreak S,$ $f_{3D}$ is
negative.
\end{theorem}

\begin{proof}
Let us take $\rho _{12}\allowbreak =\allowbreak \allowbreak q\rho _{13}\rho
_{23}$ and consider (\ref{exp_in_ASC}), then $\rho _{12}^{s}\left( \rho
_{13}\rho _{23}/\rho _{12}\right) _{s}\allowbreak =\allowbreak \left\{ 
\begin{array}{ccc}
1 & if & s=0 \\ 
-(1-q)\rho _{13}\rho _{23} & if & s=1 \\ 
0 & if & s>1%
\end{array}%
\right. .$ \newline
Hence 
\begin{gather*}
f_{3D}(x_{1},x_{2},x_{3}|q\rho _{13}\rho _{23},\rho _{13},\rho _{23},q)= \\
f_{CN}\left( x_{1}|x_{3},\rho _{13},q\right) f_{CN}\left( x_{3}|x_{2},\rho
_{23},q\right) f_{N}\left( x_{2}|q\right) (1-\frac{(1-q)\rho _{13}\rho
_{23}(x_{1}-\rho _{13}x_{3})(x_{2}-\rho _{23}x_{3})}{(1-\rho
_{13}^{2})(1-\rho _{23}^{2})}).
\end{gather*}%
Thus the sign of $f_{3D}$ is the same as the sign of 
\begin{equation*}
1-\frac{(1-q)\rho _{13}\rho _{23}(x_{1}-\rho _{13}x_{3})(x_{2}-\rho
_{23}x_{3})}{(1-\rho _{13}^{2})(1-\rho _{23}^{2})}.
\end{equation*}%
Equivalently $f_{3D}$ would be positive if for all $x_{1},x_{2},x_{3}\in
S\left( q\right) ,$ $\left\vert \rho _{13}\right\vert ,\left\vert \rho
_{23}\right\vert <1$ and 
\begin{equation*}
(1-\rho _{13}^{2})(1-\rho _{23}^{2})\geq (1-q)^{2}\rho _{13}^{2}\rho
_{23}^{2}
\end{equation*}%
which comes from the positivity in (\ref{wyzn}) we would have%
\begin{equation*}
(1-\rho _{13}^{2})(1-\rho _{23}^{2})\geq (1-q)\rho _{13}\rho
_{23}(x_{1}-\rho _{13}x_{3})(x_{2}-\rho _{23}x_{3}).
\end{equation*}%
Since $y_{i}\overset{df}{=}\sqrt{1-q}x_{i}\allowbreak \in \allowbreak
\lbrack -2,2]$ the last inequality reduces to the following :%
\begin{equation}
(1-\rho _{13}^{2})(1-\rho _{23}^{2})\geq \rho _{13}\rho _{23}(y_{1}-\rho
_{13}y_{3})(y_{2}-\rho _{23}y_{3})  \label{nier}
\end{equation}

Let us select $y_{1}\allowbreak =\allowbreak 2\rho _{13}$ and $%
y_{2}\allowbreak =\allowbreak -2\rho _{23}.$ The inequality now takes a form 
\begin{equation*}
(1-\rho _{13}^{2})(1-\rho _{23}^{2})\geq \rho _{13}^{2}\rho
_{23}^{2}(4-y_{3}^{2}).
\end{equation*}%
Now it suffices take $\rho _{13}^{2},\rho _{23}^{2}\allowbreak =\allowbreak
0.6,$ $1>q>1/3$ and $y_{3}^{2}<4-\frac{4}{9}$. On one hand one gets 
\begin{equation*}
(1-\rho _{13}^{2})(1-\rho _{23}^{2})\allowbreak =\allowbreak .16>\left( 1-%
\frac{1}{3}\right) ^{2}.6^{2}
\end{equation*}%
while on the other we have 
\begin{equation*}
\allowbreak (1-\rho _{13}^{2})(1-\rho _{23}^{2})\allowbreak \allowbreak
=\allowbreak (1-.6)^{2}\allowbreak =\allowbreak <\allowbreak .6^{2}(4-(4-%
\frac{4}{9}))\allowbreak \leq \allowbreak \rho _{13}^{2}\rho
_{23}^{2}(4-y_{3}^{2}).
\end{equation*}%
Since the function $f_{3D}$ is continuous in $x_{1},x_{2},x_{3}$ hence there
exists a neighborhood of points $\left( 2\rho _{13}/(\sqrt{1-q},2\rho _{23}/%
\sqrt{1-q},2y_{3}/\sqrt{1-q}\right) $ of positive Lebesgue measure such that
this function is negative on this neighborhood.
\end{proof}

As a corollary we get the following fact.

\begin{corollary}
\begin{gather}
\frac{f_{CN}\left( x_{1}|x_{3},\rho _{13},q\right) }{f_{CN}\left(
x_{1}|x_{2},q\rho _{13}\rho _{23},q\right) }(1-\frac{(1-q)\rho _{13}\rho
_{23}(x_{1}-\rho _{13}x_{3})(x_{2}-\rho _{23}x_{3})}{(1-\rho
_{13}^{2})(1-\rho _{23}^{2})})  \label{nonpos} \\
=\sum_{s=0}^{\infty }\frac{\rho _{13}^{s}\left( 1-q^{s}\rho _{23}^{2}\right) 
}{\left[ s\right] _{q}!\left( q^{2}\rho _{13}^{2}\rho _{23}^{2}\right)
_{s}\left( 1-\rho _{23}^{2}\right) }P_{s}\left( x_{1}|x_{2},q\rho _{13}\rho
_{23},q\right) P_{s}\left( x_{3}|x_{2},\rho _{23},q\right) .  \notag
\end{gather}
\end{corollary}

\begin{proof}
Follows the (\ref{exp_in_ASC}) and the fact that for $\rho _{12}\allowbreak
\allowbreak =\allowbreak q\rho _{13}\rho _{23}\allowbreak $ $\rho
_{13}^{s}\left( \rho _{12}\rho _{23}/\rho _{13}\right) _{s}$ $\allowbreak
=\allowbreak \rho _{13}^{s}\left( q\rho _{23}^{2}\right) _{s}$ and further
that $\frac{\left( q\rho _{23}^{2}\right) _{s}}{\left( \rho _{23}^{2}\right)
_{s}}\allowbreak =\allowbreak \frac{1-q^{s}\rho _{23}^{2}}{1-\rho _{23}^{2}}%
. $
\end{proof}

\begin{remark}
Notice that we can take $q\rho _{13}\rho _{23}$ instead $\rho _{13}\rho
_{23} $ in (\ref{rozwiniecie}) getting $\frac{f_{CN}\left( x_{1}|x_{3},\rho
_{13},q\right) }{f_{CN}\left( x_{1}|x_{2},q\rho _{13}\rho _{23},q\right) }%
\allowbreak =\allowbreak \sum_{s=0}^{\infty }\frac{\rho _{13}^{s}}{\left[ s%
\right] _{q}!\left( q^{2}\rho _{13}^{2}\rho _{23}^{2}\right) _{s}}%
P_{s}\left( x_{1}|x_{2},q\rho _{13}\rho _{23},q\right) P_{s}\left(
x_{3}|x_{2},\rho _{23},q\right) .$ Hence the positivity conditions for
kernels of the form $\sum_{s=0}^{\infty }a_{s}\left( \rho _{1},\rho
_{2}\right) P_{s}\left( x_{1}|x_{2},\rho _{1},q\right) P_{s}\left(
x_{3}|x_{2},\rho _{2},q\right) $ are quite subtle. Besides comparison of (%
\ref{rozwiniecie}) and (\ref{nonpos}) can be the source of many interesting
kernels involving Al-Salam--Chihara polynomials. In particular for $\rho
_{12}\allowbreak =\allowbreak q^{k}\rho _{13}\rho _{23}$ we have 
\begin{eqnarray*}
&&\frac{f_{CN}\left( x_{1}|x_{3},\rho _{13},q\right) }{f_{CN}\left(
x_{1}|x_{2},q^{k}\rho _{13}\rho _{23},q\right) }\sum_{j=0}^{k}\frac{%
q^{kj}\rho _{13}^{j}\rho _{23}^{j}\left( q^{-k}\right) _{j}}{\left[ j\right]
_{q}!\left( \rho _{13}^{2}\right) _{j}\left( \rho _{23}^{2}\right) _{j}}%
P_{j}\left( x_{1}|x_{3},\rho _{13},q\right) P_{j}\left( x_{2}|x_{3},\rho
_{23},q\right) \\
&=&\sum_{s=0}^{\infty }\frac{\rho _{13}^{s}\left( q^{k}\rho _{23}^{2}\right)
_{s}}{\left[ s\right] _{q}!\left( q^{2k}\rho _{13}^{2}\rho _{23}^{2}\right)
_{s}\left( \rho _{23}^{2}\right) _{s}}P_{s}\left( x_{1}|x_{2},q^{k}\rho
_{13}\rho _{23},q\right) P_{s}\left( x_{3}|x_{2},\rho _{23},q\right) .
\end{eqnarray*}
\end{remark}

\section{Open Problems\label{open}}

\begin{enumerate}
\item Consider the KS formulae for higher (than $3$) dimensions. Can it be
nonnegative for all allowed values of the parameters $\rho $ and all values
of the variables form $S\left( q\right) ?$

\item Since for $q=0$ both densities $f_{CN}$ and $f_{N}$ and ASC
polynomials (compare (\ref{q=0})) are very simple one can hope to obtain
exact formula for function $f_{3D}.$ One can deduce that then $f_{3D}$
divided by the product of the one dimensional marginals should have the form
of the ratio of quadratic forms in $3$ variables and the product of three
quadratic functions of every pair of variables. Thus one can hope to find a
set $\Theta $ in $\mathbb{(}-1,1)^{3}$ such that if $\left( \rho _{12},\rho
_{13},\rho _{23}\right) \in \Theta $ then $f_{3D}$ is a $3-$dimensional
density.

\item What about the cases $\rho _{12}=q^{k}\rho _{13}\rho _{23}$ for $k>1.$
Do we get non-positivity for all $k>1$ or are there some $k^{\prime }$s for
which it is nonnegative?

\item Numerical simulations suggest that for say $\rho _{12}\allowbreak
=\allowbreak 0$ one can find $q,$ $\rho _{13},\rho _{23},$ $%
x_{1},x_{2},x_{3}\allowbreak \in \allowbreak S\left( q\right) $ such that
one of the kernels (and consequently both) given in (\ref{rho=0}) is
negative. What would be the proof of this fact? More precisely what would be
range of parameters $q,$ $\rho _{13},\rho _{23}$ and subset $A\subset
S\left( q\right) ^{3}$ of positive Lebesgue measure such that for $%
x_{1},x_{2},x_{3}\in A$ the kernel is negative.

\item What about cases $\left\vert \rho _{12}\right\vert >\rho _{13}\rho
_{23}$ ?
\end{enumerate}

\section{Proofs\label{dowody}}

\begin{proof}[Proof of Proposition \protect\ref{proposition}]
First of all let us notice that due to inequality (\ref{ogr_H}) and
assertion viii) of Lemma \ref{pomoc} series defining function $f_{3D}$ is
absolutely convergent. Now to prove i) we apply Lemma \ref{UogCarl} i) to (%
\ref{3D1}) and (\ref{3D2}) and remembering that $\gamma _{0,0}\left(
x,y|\rho ,q\right) \allowbreak =\allowbreak f_{CN}\left( x|y,\rho ,q\right)
/f_{N}\left( x|q\right) $ we get:%
\begin{gather*}
f_{3D}\left( x_{1},x_{2},x_{3}|\rho _{12},\rho _{13},\rho _{23},q\right) \\
=\prod_{i=1}^{3}f_{N}\left( x_{i}|q\right) \sum_{j,k\geq 0}\frac{\rho
_{12}^{j}\rho _{23}^{k}}{\left[ j\right] _{q}!\left[ k\right] _{q}!}%
H_{j+k}\left( x_{2}|q\right) \sum_{i\geq 0}\frac{\rho _{13}^{i}}{\left[ i%
\right] _{q}!}H_{i+j}\left( x_{1}|q\right) H_{i+k}\left( x_{3}|q\right) \\
=f_{CN}\left( x_{1}|x_{3},\rho _{13},q\right) f_{N}\left( x_{3}|q\right)
f_{N}\left( x_{2}|q\right) \sum_{j,k\geq 0}\frac{\rho _{12}^{j}\rho _{23}^{k}%
}{\left[ j\right] _{q}!\left[ k\right] _{q}!}H_{j+k}\left( x_{2}|q\right)
Q_{j,k}(x_{1},x_{3}|\rho _{13},q).
\end{gather*}%
Now changing order of summation and substituting $j+k$\allowbreak\ $->$%
\allowbreak\ $s,$ $j->k$ we get%
\begin{gather*}
f_{3D}\left( x_{1},x_{2},x_{3}|\rho _{12},\rho _{13},\rho _{23},q\right)
=f_{CN}\left( x_{3}|x_{1},\rho _{13},q\right) f_{N}\left( x_{1}|q\right)
f_{N}\left( x_{2}|q\right) \\
\times \sum_{s\geq 0}\frac{1}{\left[ s\right] _{q}!}H_{s}\left(
x_{2}|q\right) C_{s}\left( x_{1},x_{3}|\rho _{12},\rho _{23},\rho
_{13},q\right) .
\end{gather*}%
ii) We apply (\ref{identyty2}). Then%
\begin{gather*}
f_{3D}(x_{1},x_{2},x_{3}|\rho _{12},\rho _{13},\rho _{23},q)\allowbreak
=\allowbreak \prod_{i=1}^{3}f_{N}\left( x_{i}|q\right) \sum_{i,j,k\geq 0}%
\frac{\rho _{12}^{j}\rho _{23}^{k}\rho _{13}^{i}}{\left[ i\right] _{q}!\left[
j\right] _{q}!\left[ k\right] _{q}!}H_{i+j}\left( x_{1}|q\right)
H_{j+k}\left( x_{2}|q\right) \\
\times \sum_{n\geq 0}\QATOPD[ ] {i}{n}_{q}\QATOPD[ ] {k}{n}_{q}\left[ n%
\right] _{q}!\left( -1\right) ^{n}q^{\binom{n}{2}}H_{i-n}\left(
x_{3}|q\right) H_{k-n}\left( x_{3}|q\right) \allowbreak = \\
\prod_{i=1}^{3}f_{N}\left( x_{i}|q\right) \sum_{n=0}^{\infty }\left(
-1\right) ^{n}q^{\binom{n}{2}}\frac{\rho _{23}^{n}\rho _{13}^{n}}{\left[ n%
\right] _{q}!}\sum_{j=0}^{\infty }\frac{\rho _{12}^{j}}{\left[ j\right] _{q}!%
} \\
\times \sum_{i=n}^{\infty }\sum_{k=n}^{\infty }\frac{\rho _{23}^{k-n}\rho
_{13}^{i-n}}{\left[ i-n\right] _{q}!\left[ k-n\right] _{q}!}H_{i-n}\left(
x_{3}|q\right) H_{k-n}\left( x_{3}|q\right) H_{i+j}\left( x_{1}|q\right)
H_{j+k}\left( x_{2}|q\right) \allowbreak
\end{gather*}%
\begin{gather*}
=\prod_{i=1}^{3}f_{N}\left( x_{i}|q\right) \sum_{n=0}^{\infty }\left(
-1\right) ^{n}q^{\binom{n}{2}}\frac{\rho _{23}^{n}\rho _{13}^{n}}{\left[ n%
\right] _{q}!}\sum_{j=0}^{\infty }\frac{\rho _{12}^{j}}{\left[ j\right] _{q}!%
} \\
\times \sum_{i=0}^{\infty }\sum_{k=0}^{\infty }\frac{\rho _{23}^{k}\rho
_{13}^{i}}{\left[ i\right] _{q}!\left[ k\right] _{q}!}H_{i}\left(
x_{3}|q\right) H_{i+n+j}\left( x_{1}|q\right) H_{k}\left( x_{3}|q\right)
H_{j+n+k}\left( x_{2}|q\right) .
\end{gather*}%
Now we use quantities: $\gamma _{m,k}\left( x,y|\rho ,q\right) \allowbreak
=\allowbreak \sum_{i=0}^{\infty }\frac{\rho ^{i}}{\left[ i\right] _{q}!}%
H_{i+m}\left( x|q\right) H_{i+k}\left( y|q\right) $ and apply (\ref%
{gamma_m_k}) and Lemma \ref{UogCarl} iii). We get then%
\begin{gather*}
f_{3D}(x_{1},x_{2},x_{3}|\rho _{12},\rho _{13},\rho _{23},q)\allowbreak
=\allowbreak \prod_{i=1}^{3}f_{N}\left( x_{i}|q\right) \sum_{n=0}^{\infty
}\left( -1\right) ^{n}q^{\binom{n}{2}}\frac{\rho _{23}^{n}\rho _{13}^{n}}{%
\left[ n\right] _{q}!} \\
\times \sum_{j=0}^{\infty }\frac{\rho _{12}^{j}}{\left[ j\right] _{q}!}%
\gamma _{0,n+j}\left( x_{3},x_{1}|\rho _{13},q\right) \gamma _{0,n+j}\left(
x_{3},x_{2}|\rho _{23},q\right) \allowbreak \\
=f_{CN}\left( x_{1}|x_{3},\rho _{13},q\right) f_{CN}\left( x_{2}|x_{3},\rho
_{23},q\right) f_{N}\left( x_{3}|q\right) \allowbreak \\
\times \sum_{n=0}^{\infty }\left( -1\right) ^{n}q^{\binom{n}{2}}\frac{\rho
_{23}^{n}\rho _{13}^{n}}{\left[ n\right] _{q}!}\sum_{j=0}^{\infty }\frac{%
\rho _{12}^{j}}{\left[ j\right] _{q}!\left( \rho _{13}^{2}\right)
_{n+j}\left( \rho _{23}^{2}\right) _{n+j}}P_{n+j}\left( x_{1}|x_{3},\rho
_{13},q\right) P_{n+j}\left( x_{2}|x_{3},\rho _{23},q\right) =
\end{gather*}%
\begin{gather*}
f_{CN}\left( x_{1}|x_{3},\rho _{13},q\right) f_{CN}\left( x_{2}|x_{3},\rho
_{23},q\right) f_{N}\left( x_{3}|q\right) \\
\times \sum_{s=0}^{\infty }\frac{1}{\left[ s\right] _{q}!\left( \rho
_{13}^{2}\right) _{s}\left( \rho _{23}^{2}\right) _{s}}P_{s}\left(
x_{1}|x_{3},\rho _{13},q\right) P_{s}\left( x_{2}|x_{3},\rho _{23},q\right)
\\
\times \sum_{n=0}^{s}\QATOPD[ ] {s}{n}_{q}(-1)q^{\binom{n}{2}}\left( \rho
_{13}\rho _{23}\right) ^{n}\rho _{12}^{s-n}\allowbreak =\allowbreak
f_{CN}\left( x_{1}|x_{3},\rho _{13},q\right) f_{CN}\left( x_{2}|x_{3},\rho
_{23},q\right) f_{N}\left( x_{3}|q\right) \\
\times \sum_{s=0}^{\infty }\frac{\rho _{12}^{s}\left( \rho _{13}\rho
_{23}/\rho _{12}\right) _{s}}{\left[ s\right] _{q}!\left( \rho
_{13}^{2}\right) _{s}\left( \rho _{23}^{2}\right) _{s}}P_{s}\left(
x_{1}|x_{3},\rho _{13},q\right) P_{s}\left( x_{2}|x_{3},\rho _{23},q\right) .
\end{gather*}
\end{proof}

\begin{proof}[Proof of Proposition \protect\ref{na hermity}]
We start with $C_{n}\left( x,y|\rho _{1},\rho _{2},\rho _{3},q\right)
\allowbreak =\allowbreak \sum_{i=0}^{n}\QATOPD[ ] {n}{i}_{q}\rho
_{1}^{n-i}\rho _{2}^{i}Q_{n-i},_{i}\left( x,y|\rho _{3,}q\right) \allowbreak
=\allowbreak \sum_{s=0}^{n}\QATOPD[ ] {n}{s}_{q}H_{n-s}\left( y|q\right)
P_{s}\left( x|y,\rho _{3},q\right) \rho _{1}^{n-s}\rho _{2}^{s}\left( \rho
_{1}\rho _{3}/\rho _{2}\right) _{s}/\left( \rho _{3}^{2}\right) _{s}$ and
apply Lemma \ref{pomoc} ii). 
\begin{gather*}
C_{n}\left( x_{1},x_{3}|\rho _{12},\rho _{23},\rho _{13},q\right)
\allowbreak =\allowbreak \sum_{s=0}^{n}\QATOPD[ ] {n}{s}_{q}\rho
_{12}^{n-s}H_{n-s}\left( x_{3}|q\right) \rho _{23}^{s}\left( \frac{\rho
_{12}\rho _{13}}{\rho _{23}}\right) _{s}P_{s}\left( x_{1}|x_{3},\rho
_{13},q\right) /\left( \rho _{13}^{2}\right) _{s}\allowbreak \\
=\allowbreak \frac{1}{\left( \rho _{13}^{2}\right) _{n}}\sum_{s=0}^{n}\QATOPD%
[ ] {n}{s}_{q}\rho _{12}^{n-s}H_{n-s}\left( x_{3}|q\right) \rho
_{23}^{s}\left( \frac{\rho _{12}\rho _{13}}{\rho _{23}}\right)
_{s}P_{s}\left( x_{1}|x_{3},\rho _{13},q\right) \left( \rho
_{13}^{2}q^{s}\right) _{n-s}\allowbreak \\
=\frac{1}{\left( \rho _{13}^{2}\right) _{n}}\sum_{s=0}^{n}\QATOPD[ ] {n}{s}%
_{q}\rho _{12}^{n-s}H_{n-s}\left( x_{3}|q\right) \rho _{23}^{s}\left( \frac{%
\rho _{12}\rho _{13}}{\rho _{23}}\right) _{s}\left( \rho
_{12}^{2}q^{s}\right) _{n-s}\sum_{i=0}^{s}\QATOPD[ ] {s}{i}_{q}\rho
_{13}^{s-i}B_{s-i}\left( x_{3}|q\right) H_{i}\left( x_{1}|q\right)
\allowbreak \\
=\allowbreak \frac{1}{\left( \rho _{13}^{2}\right) _{n}}\sum_{i=0}^{n}\QATOPD%
[ ] {n}{i}_{q}H_{i}\left( x_{1}|q\right) \sum_{s=i}^{n}\QATOPD[ ] {n-i}{s-i}%
_{q}\rho _{12}^{n-s}H_{n-s}\left( x_{3}|q\right) \rho _{23}^{s}\left( \frac{%
\rho _{12}\rho _{13}}{\rho _{23}}\right) _{s}\left( \rho
_{13}^{2}q^{s}\right) _{n-s}\rho _{13}^{s-i}B_{s-i}\left( x_{3}|q\right)
\allowbreak
\end{gather*}

\begin{gather*}
=\frac{1}{\left( \rho _{13}^{2}\right) _{n}}\sum_{i=0}^{n}\QATOPD[ ] {n}{i}%
_{q}H_{i}\left( x_{1}|q\right) \times \\
\sum_{j=0}^{n-i}\QATOPD[ ] {n-i}{j}_{q}\rho _{12}^{n-i-j}H_{n-i-j}\left(
x_{3}|q\right) \rho _{23}^{i+j}\left( \frac{\rho _{12}\rho _{13}}{\rho _{23}}%
\right) _{i+j}\left( \rho _{13}^{2}q^{i+j}\right) _{n-i-j}\rho
_{13}^{j}B_{j}\left( x_{3}|q\right) .
\end{gather*}

And further using formula 
\begin{equation*}
H_{m}\left( x|q\right) B_{n}\left( x|q\right) =(-1)^{n}q^{\binom{n}{2}%
}\sum_{i=0}^{\left\lfloor (n+m)/2\right\rfloor }\QATOPD[ ] {n}{i}_{q}\QATOPD[
] {n+m-i}{i}_{q}\left[ i\right] _{q}!q^{-i(n-i)}H_{n+m-2i}\left( x|q\right) ,
\end{equation*}%
proved in \cite{Szab3} (Lemma 2 i)) we get%
\begin{gather*}
C_{n}\left( x_{1},x_{3}|\rho _{12},\rho _{23},\rho _{13},q\right) =\frac{1}{%
\left( \rho _{13}^{2}\right) _{n}}\sum_{i=0}^{n}\QATOPD[ ] {n}{i}%
_{q}H_{i}\left( x_{1}|q\right) \times \\
\sum_{j=0}^{n-i}\QATOPD[ ] {n-i}{j}_{q}\rho _{12}^{n-i-j}\rho
_{23}^{i+j}\left( \frac{\rho _{12}\rho _{13}}{\rho _{23}}\right)
_{i+j}\left( \rho _{13}^{2}q^{i+j}\right) _{n-i-j}\rho _{13}^{j}(-1)^{j}q^{%
\binom{j}{2}} \\
\times \sum_{k=0}^{\left\lfloor (n-i)/2\right\rfloor }\QATOPD[ ] {j}{k}_{q}%
\QATOPD[ ] {n-i-k}{k}_{q}\left[ k\right] _{q}!q^{-k(j-k)}H_{n-i-2k}\left(
x_{3}|q\right) .
\end{gather*}%
Now keeping in mind that $\QATOPD[ ] {j}{k}_{q}\allowbreak =\allowbreak 0$
for $j<k$ we split first internal sum into two sums : one form $0$ to $%
\left\lfloor (n-i)/2\right\rfloor $ and second from $\left\lfloor
(n-i)/2\right\rfloor +1$ to $(n-i).$ 
\begin{gather*}
=\frac{1}{\left( \rho _{13}^{2}\right) _{n}}\sum_{i=0}^{n}\QATOPD[ ] {n}{i}%
_{q}H_{i}\left( x_{1}|q\right) \times (\sum_{j=0}^{\left\lfloor
(n-i)/2\right\rfloor }\QATOPD[ ] {n-i}{j}_{q}\rho _{12}^{n-i-j}\rho
_{23}^{i+j}\left( \frac{\rho _{12}\rho _{13}}{\rho _{23}}\right)
_{i+j}\left( \rho _{13}^{2}q^{i+j}\right) _{n-i-j}\rho _{13}^{j}(-1)^{j}q^{%
\binom{j}{2}} \\
\times \sum_{k=0}^{j}\QATOPD[ ] {j}{k}_{q}\QATOPD[ ] {n-i-k}{k}_{q}\left[ k%
\right] _{q}!q^{-k(j-k)}H_{n-i-2k}\left( x_{3}|q\right) \\
+\sum_{j=\left\lfloor (n-i)/2\right\rfloor +1}^{(n-i)}\QATOPD[ ] {n-i}{j}%
_{q}\rho _{12}^{n-i-j}\rho _{23}^{i+j}\left( \frac{\rho _{12}\rho _{13}}{%
\rho _{23}}\right) _{i+j}\left( \rho _{13}^{2}q^{i+j}\right) _{n-i-j}\rho
_{13}^{j}(-1)^{j}q^{\binom{j}{2}} \\
\times \sum_{k=0}^{\left\lfloor (n-i)/2\right\rfloor }\QATOPD[ ] {j}{k}_{q}%
\QATOPD[ ] {n-i-k}{k}_{q}\left[ k\right] _{q}!q^{-k(j-k)}H_{n-i-2k}\left(
x_{3}|q\right) ).
\end{gather*}

We have further after changing the order of summation.%
\begin{gather*}
C_{n}\left( x_{1},x_{3}|\rho _{12},\rho _{23},\rho _{13},q\right) =\frac{1}{%
\left( \rho _{13}^{2}\right) _{n}}\sum_{i=0}^{n}\QATOPD[ ] {n}{i}%
_{q}H_{i}\left( x_{1}|q\right) \\
\times (\sum_{k=0}^{\left\lfloor (n-i)/2\right\rfloor }\frac{\left[ n-i%
\right] _{q}!}{\left[ k\right] _{q}!\left[ n-i-2k\right] _{q}!}%
H_{n-i-2k}\left( x_{3}|q\right) \\
\times \sum_{j=k}^{\left\lfloor (n-i)/2\right\rfloor }q^{-k(j-k)}\rho
_{13}^{j}(-1)^{j}q^{\binom{j}{2}}\frac{\left[ n-i-k\right] _{q}!}{\left[ j-k%
\right] _{q}!\left[ n-i-j\right] _{q}!}\rho _{12}^{n-i-j}\rho
_{23}^{i+j}\left( \frac{\rho _{12}\rho _{13}}{\rho _{23}}\right)
_{i+j}\left( \rho _{13}^{2}q^{i+j}\right) _{n-i-j} \\
+\sum_{k=0}^{\left\lfloor (n-i)/2\right\rfloor }\frac{\left[ n-i\right] _{q}!%
}{\left[ k\right] _{q}!\left[ n-i-2k\right] _{q}!}H_{n-i-2k}\left(
x_{3}|q\right) \\
\times \sum_{j=\left\lfloor (n-i)/2\right\rfloor +1}^{(n-i)}\frac{\left[
n-i-k\right] _{q}!}{\left[ j-k\right] _{q}!\left[ n-i-j\right] _{q}!}%
(-1)^{j}q^{-k(j-k)}q^{\binom{j}{2}}\rho _{13}^{j}\rho _{12}^{n-i-j}\rho
_{23}^{i+j}\left( \frac{\rho _{12}\rho _{13}}{\rho _{23}}\right)
_{i+j}\left( \rho _{13}^{2}q^{i+j}\right) _{n-i-j})
\end{gather*}

\begin{gather*}
=\frac{1}{\left( \rho _{13}^{2}\right) _{n}}\sum_{i=0}^{n}\QATOPD[ ] {n}{i}%
_{q}H_{i}\left( x_{1}|q\right) \\
\times (\sum_{k=0}^{\left\lfloor (n-i)/2\right\rfloor }\frac{\left[ n-i%
\right] _{q}!}{\left[ k\right] _{q}!\left[ n-i-2k\right] _{q}!}%
H_{n-i-2k}\left( x_{3}|q\right) \\
\times \sum_{j=k}^{(n-i)}\QATOPD[ ] {n-i-k}{j-k}_{q}(-1)^{j}q^{-k(j-k)}q^{%
\binom{j}{2}}\rho _{13}^{j}\rho _{12}^{n-i-j}\rho _{23}^{i+j}\left( \frac{%
\rho _{12}\rho _{13}}{\rho _{23}}\right) _{i+j}\left( \rho
_{13}^{2}q^{i+j}\right) _{n-i-j}).
\end{gather*}

Hence%
\begin{gather*}
C_{n}\left( x_{1},x_{3}|\rho _{12},\rho _{23},\rho _{13},q\right) =\frac{1}{%
\left( \rho _{13}^{2}\right) _{n}}\sum_{i=0}^{n}\QATOPD[ ] {n}{i}%
_{q}H_{i}\left( x_{1}|q\right) \\
\times \sum_{k=0}^{\left\lfloor (n-i)/2\right\rfloor }(-1)^{k}q^{\binom{k}{2}%
}\frac{\left[ n-i\right] _{q}!}{\left[ k\right] _{q}!\left[ n-i-2k\right]
_{q}!}\rho _{13}^{k}H_{n-i-2k}\left( x_{3}|q\right) \\
\times \sum_{s=0}^{n-i-k}\QATOPD[ ] {n-i-k}{s}_{q}(-1)^{s}q^{\binom{s}{2}%
}\rho _{13}^{s}\rho _{12}^{n-i-k-s}\rho _{23}^{i+k+s}\left( \frac{\rho
_{12}\rho _{13}}{\rho _{23}}\right) _{i+k+s}\allowbreak \left( \rho
_{13}^{2}q^{i+k+s}\right) _{n-i-k-s}.
\end{gather*}

So 
\begin{gather*}
C_{n}\left( x_{1},x_{3}|\rho _{12},\rho _{13},\rho _{23},q\right)
=\allowbreak \frac{1}{\left( \rho _{13}^{2}\right) _{n}}\sum_{i=0}^{n}\QATOPD%
[ ] {n}{i}_{q}H_{i}\left( x_{1}|q\right) \\
\times \sum_{k=0}^{\left\lfloor (n-i)/2\right\rfloor }(-1)^{k}q^{\binom{k}{2}%
}\frac{\left[ n-i\right] _{q}!}{\left[ k\right] _{q}!\left[ n-i-2k\right]
_{q}!}\rho _{13}^{k}\rho _{23}^{k+i}H_{n-i-2k}\left( x_{3}|q\right) \left( 
\frac{\rho _{12}\rho _{13}}{\rho _{23}}\right) _{k+i} \\
\times \sum_{s=0}^{n-i-k}\QATOPD[ ] {n-i-k}{s}_{q}(-1)^{s}q^{\binom{s}{2}%
}\rho _{13}^{s}\rho _{12}^{n-i-k-s}\rho _{23}^{s}\left( \frac{\rho _{12}\rho
_{13}}{\rho _{23}}q^{i+k}\right) _{s}\left( \rho _{13}^{2}q^{i+k+s}\right)
_{n-i-k-s}.
\end{gather*}%
Now we use formula $\sum_{i=0}^{n}\left( -1\right) ^{i}q^{\binom{i}{2}}%
\QATOPD[ ] {n}{i}_{q}\left( a\right) _{i}b^{i}\left( abq^{i}\right)
_{n-i}\allowbreak =\allowbreak \left( b\right) _{n}$ proved in (\cite{Szab3}%
) (Lemma 1 ii)) with $a\allowbreak =\allowbreak \frac{\rho _{12}\rho _{13}}{%
\rho _{23}}q^{i+k},$ $b\allowbreak =\allowbreak \rho _{13}\rho _{23}/\rho
_{12}$ getting 
\begin{gather*}
C_{n}\left( x_{1},x_{3}|\rho _{12},\rho _{13},\rho _{23},q\right)
=\allowbreak \frac{1}{\left( \rho _{13}^{2}\right) _{n}}\sum_{i=0}^{n}\QATOPD%
[ ] {n}{i}_{q}H_{i}\left( x_{1}|q\right) \\
\times \sum_{k=0}^{\left\lfloor (n-i)/2\right\rfloor }(-1)^{k}q^{\binom{k}{2}%
}\frac{\left[ n-i\right] _{q}!}{\left[ k\right] _{q}!\left[ n-i-2k\right]
_{q}!}\rho _{13}^{k}\rho _{23}^{k+i}H_{n-i-2k}\left( x_{3}|q\right) \left( 
\frac{\rho _{12}\rho _{13}}{\rho _{23}}\right) _{k+i}\rho
_{12}^{n-i-k}\left( \frac{\rho _{12}\rho _{23}}{\rho _{12}}\right) _{n-k-i}
\end{gather*}

Finally change the order of summation and using on the way an obvious
property of $\left( a\right) _{n}\allowbreak =\allowbreak \left( a\right)
_{j}\left( aq^{j}\right) _{n-j}$ for every $0\leq j\leq n$ 
\begin{gather*}
C_{n}\left( x_{1},x_{3}|\rho _{12},\rho _{13},\rho _{23},q\right) =\frac{1}{%
\left( \rho _{13}^{2}\right) _{n}}\sum_{k=0}^{\left\lfloor n/2\right\rfloor
}(-1)^{k}q^{\binom{k}{2}}\QATOPD[ ] {n}{2k}_{q}\QATOPD[ ] {2k}{k}_{q}\left[ k%
\right] _{q}!\rho _{12}^{k}\rho _{23}^{k}\rho _{13}^{k}\left( \frac{\rho
_{12}\rho _{13}}{\rho _{23}}\right) _{k}\left( \frac{\rho _{12}\rho _{23}}{%
\rho _{12}}\right) _{k} \\
\times \sum_{i=0}^{n-2k}\frac{\left[ n-2k\right] _{q}!}{\left[ n-i-2k\right]
_{q}!\left[ i\right] _{q}!}\rho _{23}^{i}\left( \frac{\rho _{12}\rho _{13}}{%
\rho _{23}}q^{k}\right) _{i}\rho _{12}^{n-i-2k}\left( \frac{\rho _{13}\rho
_{23}}{\rho _{12}}q^{k}\right) _{n-i-2k}H_{i}\left( x_{1}|q\right)
H_{n-2k-i}\left( x_{3}|q\right) .
\end{gather*}%
$\allowbreak \allowbreak $
\end{proof}

\end{document}